\documentclass[reqno]{amsart}%
\usepackage{palatino, mathpazo}
\usepackage{amsfonts}
\usepackage{amsmath}
\usepackage{amssymb,latexsym,xcolor}
\usepackage{graphicx}
\usepackage{harpoon}
\usepackage[mathscr]{eucal}
\usepackage{amssymb}%
\usepackage[%backref=page,
linktocpage=true,colorlinks,citecolor=magenta,linkcolor=blue,urlcolor=magenta]{hyperref}
\hypersetup{
 colorlinks=true,
 linkcolor=blue,
 citecolor=blue,
 urlcolor=blue,
 pdftitle={Exact counting of spherical metrics with one conical singularity on rectangular tori},
 pdfsubject={Exact counting of spherical metrics with one conical singularity on rectangular tori},
 pdfkeywords={Rectangular torus, spherical metric, Lam\'e equation,
 Hill discriminant, commutator trace formula}
}

\providecommand{\U}[1]{\protect \rule{.1in}{.1in}}

\newtheorem{theorem}{Theorem}[section]

\newtheorem{corollary}[theorem]{Corollary}

\newtheorem{lemma}[theorem]{Lemma}
\newtheorem{proposition}[theorem]{Proposition}
\newtheorem{Theorem}{Theorem}

\theoremstyle{remark}
\newtheorem{remark}[theorem]{Remark}

\numberwithin{equation}{section}

\newcommand{\C}{\mathbb C}
\newcommand{\R}{\mathbb R}
\newcommand{\N}{\mathbb N}
\newcommand{\ii}{\mathrm i}
\newcommand{\e}{\mathrm e}
\newcommand{\tr}{\operatorname{tr}}
\newcommand{\dd}{\,\mathrm d}

\begin{document}
\title[Exact counting of spherical metrics]{Exact counting of spherical metrics with one conical singularity on rectangular tori}
%{A rectangular torus admits exactly $n$ spherical metrics with a cone singularity of angle  $2\pi\vartheta$ for any $\vartheta\in (2n-1, 2n+1)$}

\author{Zhijie Chen}
\address{Department of Mathematical Sciences, Yau Mathematical Sciences Center,
Tsinghua University, Beijing, 100084, China }
\email{zjchen2016@tsinghua.edu.cn}
\author{Shihong Zhang}
\address{Yau Mathematical Sciences Center,
Tsinghua University, Beijing, 100084, China}
\email{shihong-zhang@tsinghua.edu.cn}

%\date{August 2026}

\begin{abstract}
We prove that for every integer $n\geq 2$ and
$8\pi(n-1)<\rho<8\pi n$, the singular Liouville equation
$\Delta u+\e^u=\rho\delta_0$ on a rectangular torus $E_{\ii b}=\mathbb{C}/(\mathbb Z+\ii b\mathbb Z)$ has exactly $n$ solutions, which are all axisymmetric. 
Together with previous results by Chen-Lin and Lin-Wang, this yields that 
\begin{itemize}
\item $E_{\ii b}$ admits no spherical metrics with a conical singularity of angle $2\pi\vartheta$ as long as $\vartheta$ is a positive odd integer.
\item For every integer $n\geq 1$, $E_{\ii b}$ admits exactly $n$ spherical metrics with a conical singularity of angle $2\pi\vartheta$ for each $\vartheta\in (2n-1, 2n+1)$.
\end{itemize}
The basic idea is to prove that the linearized equation
has only trivial solutions in the space of axisymmetric functions. The previous method of analyzing nodal domains via Bol's isoperimetric inequality only works for $\rho\leq 8\pi$. We develop a unified approach for all $\rho\in (0,+\infty)\setminus 8\pi\mathbb{N}_{\geq 1}$ by exploring the deep connection with the monodromy of the classical Lam\'{e} equation.
\end{abstract}

\maketitle

\section{Introduction and main results}\label{sec:introduction}

 \subsection{Main results}
Let $E_{\tau}:=\mathbb{C}/(\mathbb Z+\mathbb Z\tau)$ be a flat torus in the plane, where $\tau \in\mathbb{C}$ with $\operatorname{Im}\tau>0$. Consider the following singular Liouville equation with parameter $\rho>0$
\begin{equation}\label{eq:mfe}
 \Delta u+e^u=\rho\delta_0\qquad\text{on }E_{\tau}.
\end{equation}
Here we use the complex variable $z=x+iy$, $\Delta=\partial_x^2+\partial_y^2=4\partial^2_{\bar z z}$ is the Laplace operator, and $\delta_{0}$ is the Dirac measure at $0$. Define
\begin{equation}\label{eq:parameters}
 \eta:=\frac{\rho}{8\pi},\qquad
 \vartheta:=2\eta+1.
\end{equation}

Equation \eqref{eq:mfe} is related to various research subjects. First,
 \eqref{eq:mfe} has its origin in the prescribed curvature problem in conformal geometry. Given any solution $u(z)$ of \eqref{eq:mfe}, the new metric
\begin{equation}\label{eq:metric}
 g=\frac12e^u|\dd z|^2
\end{equation}
has constant Gaussian curvature $K_g=1$ on $E_{\tau}\setminus\{0\}$, hence is a spherical metric with a conical singularity at $0$ of angle $2\pi\vartheta$.  Spherical metrics with conical singularities on compact Riemann surfaces have received great interest for a long time; see e.g. \cite{BMM, Kuo, MZ-IMRN, MP-IMRN, MP-GAFA, WWX} and the references therein for recent developments on this subject. Equation \eqref{eq:mfe} is also a special case of the following general mean field equations on a compact Riemann surface $M$:
\begin{equation}\label{eq-01}
\Delta u+\rho\left(\frac{he^u}{\int he^u}-\frac{1}{|M|}\right)=4\pi\sum_{j=1}^n\alpha_j\left(\delta_{p_j}-\frac{1}{|M|}\right),
\end{equation}
where $h$ is a positive function on $M$, $\alpha_j>-1$ for all $j$ and $\rho>0$. Equation \eqref{eq-01} arises not only from geometry, but also from
 statistical physics as the equation for
the mean field limit of the Euler flow in Onsager's vortex model (cf.
\cite{CLMP}), hence its name. Equation \eqref{eq-01} has been widely studied, and we refer the readers to \cite{BCLT, BartolucciTarantelloCMP, CM, ChenLinCPAM, DKM, JWY, MR, Tarantello} and the references therein for this subject. Besides, equation \eqref{eq:mfe} was shown to be related to the self-dual condensation of the Chern-Simons-Higgs model (cf. \cite{CLW4, LY, NT1}).

The existence and structure of solutions for equation \eqref{eq:mfe} is very challenging from
the viewpoints of geometry and PDE,
and has been studied extensively by Chai-Lin-Wang \cite{ChaiLinWang}, Chen-Lin \cite{ChenLinAJM, ChenLinJDG}, Chen-Lin-Wang \cite{CLW4}, Eremenko \cite{Eremenko}, Eremenko-Gabrielov \cite{EG}, Eremenko-Mondello-Panov \cite{EMP}, Lin \cite{LinJDG} and Lin-Wang \cite{LinWangAnnals, LinWang, LinWangJEP}. 
Note that
\begin{equation}\label{ulog}
 u(z)=4\eta\log|z|+v(z)\quad\text{near $z=0$,}
\end{equation}
where $v$ is $C^2$ at $0$ by the elliptic regularity theory \cite{BM,GilbargTrudinger}. Let \(G(z)\) be the Green function on $E_{\tau}$ with the singularity at $0$, defined by
\begin{equation}\label{eq:Green}
 -\Delta G(z)=\delta_0-\frac{1}{|E_{\tau}|},\qquad \int_{E_{\tau}}G\dd x\dd y=0.
\end{equation}
Then
\[
 G(z)=-\frac1{2\pi}\log|z|+R(z)\quad\text{near $z=0$}
\]
with \(R\) smooth. Thus, $u+\rho G$ is $C^2$ on $E_\tau$.
By applying the compactness result \cite[Corollary 4]{BartolucciTarantelloCMP} of Bartolucci-Tarantello to \eqref{eq:mfe}, it is known that solutions of \eqref{eq:mfe} have uniform a priori estimates in $C^2(E_{\tau})$.

\begin{Theorem}\label{thm A} \cite[Corollary 4]{BartolucciTarantelloCMP}
For any closed interval $K\Subset (0,+\infty)\setminus 8\pi \mathbb{N}_{\geq 1}$, there exists a constant $C_{K}$ such that for any solution $u(z)$ of \eqref{eq:mfe} with $\rho\in K$, 
$$\big\|u+\rho G\big\|_{C^2(E_{\tau})}\leq C_{K}.$$
\end{Theorem}

Let $n\in\mathbb{N}_{\geq 1}$ and $\rho\in (8\pi (n-1), 8\pi n)$. Then by Theorem \ref{thm A}, the topological Leray-Schauder degree $d_{\rho}$ of \eqref{eq:mfe} is well-defined. Chen-Lin \cite[Theorem 1.1]{ChenLinCPAM} proved an explicit degree counting formula for the general mean field equation \eqref{eq-01}, which has the following consequence for \eqref{eq:mfe}.

\begin{Theorem}\cite[Theorem 1.1]{ChenLinCPAM} \label{thm B}
Let $n\in\mathbb{N}_{\geq 1}$ and $\rho\in (8\pi (n-1), 8\pi n)$. Then $d_{\rho}=n>0$, so \eqref{eq:mfe} always has solutions.
\end{Theorem}

On the other hand, the critical case $\rho\in 8\pi \mathbb{N}_{\geq 1}$ is quite different, and whether solutions of \eqref{eq:mfe} exist or not
essentially depends on the geometry of $E_{\tau}$, or equivalently, the choice of $\tau$. This
phenomenon was first discovered by Lin-Wang \cite{LinWangAnnals} when
they studied 
\begin{equation}
\Delta u+e^{u}=8\pi \delta_{0}\  \  \text{on}\ E_{\tau}. \label{eq1-1}%
\end{equation}
Note that the three half periods of the torus are always critical points of the Green function $G(z)$ because $G(z)$ is even.
Lin-Wang \cite{LinWangAnnals} proved that $G(z)$ has either three or five critical points, and
\eqref{eq1-1} has solutions if and only if $G(z)$ has five critical points.
For example, when $\tau \in \ii\mathbb{R}_{>0}$, i.e., $E_{\tau}$
is a rectangular torus, equation \eqref{eq1-1} has no solutions;
while for $\tau=\frac{1}{2}+\frac{\sqrt{3}}{2}\ii$, i.e., $E_{\tau}$ is a rhombus
torus, equation \eqref{eq1-1} has infinitely many solutions. 
For the general case $\rho=8\pi n$ with $n\in\mathbb{N}_{\geq 2}$, 
\begin{equation}
\Delta u+e^{u}=8n\pi \delta_{0}\  \  \text{on}\ E_{\tau}, \label{eq0-2}%
\end{equation}
Chai-Lin-Wang \cite{ChaiLinWang} and subsequently Lin-Wang \cite{LinWangJEP} studied \eqref{eq0-2}
from the viewpoint of algebraic geometry, and developed a theory to connect
\eqref{eq0-2} with hyperelliptic curves and modular forms. In particular, they proved that once \eqref{eq0-2} has a solution, then \eqref{eq0-2} has infinitely many solutions. Define
\begin{align}\label{eq-en}\mathcal{E}_n:=
\{\tau\in\mathbb C\;:\; \operatorname{Im}\tau>0,\; \text{equation \eqref{eq0-2} has solutions}\}.
\end{align}
Recently, Lin \cite{LinJDG} proved that $\mathcal{E}_n\neq \emptyset$ is an unbounded open set for any $n\in\mathbb N_{\geq 1}$.

On the other hand, Lin-Wang proved in \cite[Theorem 1.4]{LinWang} that the solution of \eqref{eq:mfe} is unique for any $\rho\in (0,8\pi)$. A natural but challenging problem is how many solutions \eqref{eq:mfe} might have when $\rho\in (8\pi (n-1), 8\pi n)$ for $n\in\mathbb{N}_{\geq 2}$.   
Theorem \ref{thm B} shows that the Leray-Schauder degree $d_{\rho}=n$ for such $\rho$. However, generally the degree does not reflect the actual number of solutions. For example, for $\tau\in\mathcal{E}_1$, Lin-Wang  proved in \cite[Theorem 1.4]{LinWang} that \eqref{eq:mfe} has four solutions for $\rho-8\pi>0$ small (Note the degree $d_{\rho}=2$ for such $\rho$), while Chai-Lin-Wang proved in \cite[Corollary 3.5.1]{ChaiLinWang} that \eqref{eq:mfe} with $\rho=12\pi$ has at most two solutions (and even has only one solution for some tori). These results indicate that the bifurcation diagram of \eqref{eq:mfe} is already very complicate for $\rho$ varying from $8\pi$ to $12\pi$.
Therefore, to count the exact number of solutions of \eqref{eq:mfe} seems impossible for general torus when $\rho\in (8\pi (n-1), 8\pi n)$ with $n\in\mathbb{N}_{\geq 2}$. The best result on this aspect is \cite[Corollary 2]{Eremenko}, where Eremenko proved that the number of solutions is always finite for arbitrary torus and $\rho\in (0,+\infty)\setminus8\pi\mathbb N_{\geq 1}$.

The main result of this paper is the following result concerning the exact number of solutions for \eqref{eq:mfe} when $E_{\tau}$ is a rectangular torus.

\begin{theorem}\label{thm:main} Let $\tau=\ii b$ with $b>0$, i.e., $E_{\tau}=E_{\ii b}$ is a rectangular torus. 
Let \(n\in\N_{\geq 2}\).  Then for every
\(\rho\in(8\pi(n-1),8\pi n)\), equation \eqref{eq:mfe} has exactly
\(n\) solutions. Furthermore, all solutions are axisymmetric (i.e., invariant under the rectangular reflections):
$$u(z)=u(-z)=u(\bar z).$$ 
\end{theorem}

Together with previous results from Lin-Wang \cite{LinWangAnnals, LinWang} and Chen-Lin \cite{ChenLinAJM}, we obtain the following complete description of the solution number of \eqref{eq:mfe} for rectangular tori.

\begin{corollary}\label{coro:main}
Let $\tau=\ii b$ with $b>0$, i.e., $E_{\tau}=E_{\ii b}$ is a rectangular torus. 
\begin{itemize}
\item[(1)] If $\rho\in 8\pi\mathbb N_{\geq 1}$, then \eqref{eq:mfe} has no solutions. Equivalently, $E_{\ii b}$ admits no spherical metrics with a conical singularity of angle $2\pi\vartheta$ as long as $\vartheta$ is a positive odd integer.
\item[(2)] If $n\in\mathbb{N}_{\geq 1}$ and \(\rho\in(8\pi(n-1),8\pi n)\), then \eqref{eq:mfe} has exactly
\(n\) solutions. Equivalently, $E_{\ii b}$ admits exactly $n$ spherical metrics with a conical singularity of angle $2\pi\vartheta$ for each $\vartheta\in (2n-1, 2n+1)$.
\end{itemize}

\end{corollary}

\begin{proof}
The assertion (1) was proved by Lin-Wang \cite{LinWangAnnals} for $n=1$ and Chen-Lin \cite[Theorem 1.1]{ChenLinAJM} for $n\geq 2$. For the assertion (2), the case $n=1$ was proved in \cite[Theorem 1.4]{LinWang}, and the case $n\geq 2$ follows from Theorem \ref{thm:main}.
\end{proof}

\begin{remark}\label{rmk1-3} Here are some comments for Theorem \ref{thm:main} and Corollary \ref{coro:main}.
\begin{itemize}
\item[(1)] The previous argument already indicates that the condition of $E_{\tau}=E_{\ii b}$ being a rectangular torus is essential for the validity of Theorem \ref{thm:main} and Corollary \ref{coro:main}, namely we can not expect such statements for arbitrary tori.
\item[(2)] Theorem \ref{thm:main} was proved by Chai-Lin-Wang \cite[Corollary 0.4.2]{ChaiLinWang} for $\rho=8\pi n-4\pi$, and by Chen-Lin \cite[Theorem 1.3]{ChenLinJDG} for \begin{equation}\label{rhoend}\rho\in (8\pi(n-1), 8\pi(n-1)+\varepsilon_{b,n})\cup (8\pi n-\varepsilon_{b,n}, 8\pi n),\end{equation} where $\varepsilon_{b,n}>0$ is a small unknown constant depending on $\tau=\ii b$ and $n$. The approach of \cite[Corollary 0.4.2]{ChaiLinWang} will be explained in Remark \ref{rmk1-5} below, which only work for $\rho=8\pi n-4\pi$. The approach of \cite[Theorem 1.3]{ChenLinJDG} essentially relies on the bubbling phenomena for $\rho\to 8\pi n$ or $8\pi (n-1)$ and hence only work for $\rho$ close to these two values. 
Whether Theorem \ref{thm:main} holds for all $\rho\in (8\pi(n-1), 8\pi n)$ remained open as a conjecture there; see \cite[Conjecture C]{ChenLinJDG}. Theorem \ref{thm:main} solves \cite[Conjecture C]{ChenLinJDG} completely. 
\end{itemize}
\end{remark}

Corollary \ref{coro:main} has an interesting application to the classical Lam\'{e} equation
\begin{equation}\label{eq:Lame1}
 Y''(z)=\bigl[\eta(\eta+1)\wp(z)+B\bigr]Y(z),
\end{equation}
where the parameter $\eta>0$ is called the index and $B\in\mathbb C$ is called the accessory parameter, and $\wp(z)=\wp(z;\ii b)$ is the Weierstrass $\wp$-function with double periods $1$ and $\ii b$. see e.g. \cite{ChaiLinWang} for detailed introductions on this equation.

By the Liouville theorem, it is easy to prove that if $u(z)$ is a solution of \eqref{eq:mfe}, then there is an accessory parameter $B_0\in\mathbb C$ such that 
\begin{equation}\label{eqB0}u_{zz}-\frac12 u_z^2=-2[\eta(\eta+1)\wp(z)+B_0],\end{equation}
and moreover, the monodromy of \eqref{eq:Lame1} with $B=B_0$ is unitarizable, namely up to a common conjugation, the monodromy group of \eqref{eq:Lame1} is a subgroup of $\mathrm{SU}(2)$.
We will briefly review this fact in Section \ref{Sec 4}. In general, the monodromy of \eqref{eq:Lame1} is difficult to determine when $\eta\notin\frac12\mathbb Z$. Chen-Lin proved in \cite[Theorem 1.6]{ChenLinJDG}
that if $\eta=n\in\mathbb{N}_{\geq1}$, then the monodromy of \eqref{eq:Lame1} can not be unitarizable for any $B$.
Here we have the following result for $\eta\notin\mathbb N_{\geq 1}$. 

\begin{theorem}\label{thm:Lame}
Let $n\in\mathbb{N}_{\geq 1}$ and $\eta\in (n-1,n)$. Then there are exactly $n$ $B$'s such that the monodromy of \eqref{eq:Lame1} is unitarizable.
\end{theorem}

\begin{remark}\label{rmk1-5}
Theorem \ref{thm:Lame} for $n=1$ follows directly from Lin-Wang's uniqueness result \cite[Theorem 1.4]{LinWang} for $\rho\in (0,8\pi)$.
Theorem \ref{thm:Lame} for $\eta=n-\frac12$ was proved in \cite[Corollary 0.4.2]{ChaiLinWang}. Since $\eta=n-\frac{1}{2}$ is a half integer, the local exponent difference $2\eta+1$ is an \emph{even integer}, which infers the existence of a polynomial $P_n(B)$ of degree $n$ in $B$ such that solutions of \eqref{eq:Lame1} with $\eta=n-\frac{1}{2}$ have no logarithmic singularities if and only if $P_n(B)=0$. It was proved in \cite[Corollary 0.4.2]{ChaiLinWang} that, (i) for each such $B$, the monodromy is unitarizable and indeed the projective monodromy group is the Klein four group; (ii) there is a one-to-one correspondence between solutions of \eqref{eq:mfe} with $\rho=8\pi (n-\frac{1}{2})$ and zeros of $P_n(B)$. Thus the conclusions of Theorem \ref{thm:Lame} for $\eta=n-\frac12$ and Theorem \ref{thm:main} for $\rho=8\pi (n-\frac{1}{2})$  follow from the fact that $P_n$ has exactly $n$ different roots because $\tau=\ii b$. Clearly this idea can not apply for $\rho\neq8\pi (n-\frac{1}{2})$ due to the essential difference between $\eta\notin\frac{1}{2}\mathbb{Z}$ and $\eta\in\frac{1}{2}\mathbb{Z}$ for the Lam\'{e} equation.
\end{remark}

\subsection{Ideas of the proof}
Now we explain the main ideas of the proof, together with a brief description on the content of each section.

Thanks to the previous results recalled in Remark \ref{rmk1-3}-(2), we will use the continuity method to prove Theorem \ref{thm:main}.
The proof consists of three main ingredients: the axisymmetry of all
solutions, the nondegeneracy in the axisymmetric class, and the
continuation of the solution count across the whole parameter interval.

First,
Lin--Wang already proved that every solution of \eqref{eq:mfe} is even whenever
\(\rho\notin8\pi\N_{\geq 1}\); see \cite[Lemma~4.3]{LinWang}.  It remains to prove the
reflection symmetry across a rectangular axis.  The unitarizability of the monodromy of 
the Lam\'e equation forces the accessory parameter
\(B_0\) in \eqref{eqB0} to be real, by using the Floquet theory for the associated real Hill
operator that will be briefly recalled in Section \ref{sec:floquet}.  Comparing the projective connections of \(u(z)\) and
\(u(\bar z)\), and using the non-coaxiality of the monodromy due to $\eta\notin\mathbb Z$, then gives
\(u(\bar z)=u(z)\).  Thus every solution is axisymmetric. See Lemma \ref{symmetric lem}.

Our second ingredient is to prove the following nondegeneracy in the axisymmetric class.

\begin{theorem}[=Theorem \ref{thm:ax-nondeg}]\label{thm:zeroj}
Let $n\in\mathbb{N}_{\geq 1}$, $\rho\in (8\pi(n-1), 8\pi n)$ and $u(z)$ be a solution of \eqref{eq:mfe}, which is axisymmetric. Then 
\begin{equation}\label{eq:Jacobi}
\begin{cases}
\Delta \phi+e^u\phi=0\quad\text{on }E_{\ii b}\\
\phi(z)=\phi(-z)=\phi(\bar z)
\end{cases}
\end{equation}
has only trivial solutions.
\end{theorem}

Before introducing our new unified approach of proving this result for all $\rho\in (0,+\infty)\setminus8\pi\mathbb{N}_{\geq 1}$, we want to explain why the method of analyzing nodal domains via Bol's isoperimeter inequality, which has been widely used in the literature, can not work here. Let us take the reference \cite{LinWangAnnals} for example. Suppose $u(z)$ is an even solution of \eqref{eq:mfe}, Lin-Wang proved in \cite[Section 4]{LinWangAnnals} that 
\begin{equation}\label{eq:Jacobi1}
\begin{cases}
\Delta \phi+e^u\phi=0\quad\text{on }E_{\tau}\\
\phi(z)=\phi(-z)
\end{cases}
\end{equation}
has only trivial solutions for $\rho\in [4\pi,8\pi]$, where $E_{\tau}$ can be arbitrary torus. Their method is as follows. Suppose $\phi\neq 0$ is a solution of \eqref{eq:Jacobi1}, then $\phi$ changes sign and hence has at least two nodal domains. By using Bol's isoperimetric inequality, they can prove that the integral of $e^u$ on each simply-connected nodal domain is at least $4\pi$. 
Then, they can prove that
$
\int_{E_{\tau}}e^u>8\pi,$
which contradicts \(\rho\leq 8\pi\). 
However, determining the precise number of nodal domains appears virtually impossible in our problem, since \(\rho\) can be arbitrarily large.

Our new unified approach explores the surprising effect of a nontrivial axisymmetric Jacobi field $\phi$ on the monodromy of the Lam\'{e} equation. The first key point of our second ingredient is Theorem \ref{thm:transversality}.  Combining the commutator trace formula established in Section \ref{Sec 3} with the strict log-concavity of the horizontal and vertical
Hill stability factors established in Section \ref{sec:floquet}, we show that every real unitarizable Lam\'e
value $B_0$ is transverse to the unitarizable character locus.  In
particular, the mixed monodromy trace has a nonzero first variation
along the real accessory-parameter direction:
\begin{equation}\label{eq:mixed-transverse0}
 \frac{\dd}{\dd B}\mathrm{Im}\tr(M_h(B)M_v(B))\bigg|_{B=B_0}\ne0.
\end{equation}
Here $M_h(B)$ and $M_v(B)$ are the monodromy matrices of the Lam\'e equation \eqref{eq:Lame1} with respect to the basic cycles $z\to z+1$ and $z\to z+\ii b$ respectively.

Section \ref{sec:axisymmetric-nondegeneracy} converts this transversality into the nondegeneracy.  
A nontrivial
axisymmetric Jacobi field \(\phi\) would determine a nonzero real constant $\phi_{zz}-u_z\phi_z=c\in\mathbb{R}\setminus\{0\}$, and hence a nonzero variation $B(t)=B_0-\frac{c}{2}t$
of the accessory parameter $B_0$ along $\mathbb R$. Since
$$(u+t\phi)_{zz}-\frac12(u+t\phi)_z^2=-2[\eta(\eta+1)\wp(z)+B(t)]+O(t^2),$$
and $\phi$ is single-valued and doubly periodic,
Theorem \ref{Key thm} shows, however, that the
corresponding exact Lam\'e deformation (i.e., consider $Y"=[\eta(\eta+1)\wp(z)+B(t)]Y$) is tangent, to first order, to
the spherical monodromy locus:  
$$\frac{\dd}{\dd t}\operatorname{Im}\tr(M_h(B(t))M_v(B(t)))\bigg|_{t=0}=0.$$
This contradicts \eqref{eq:mixed-transverse0}.
Therefore no nontrivial axisymmetric Jacobi field exists, and hence completes the proof of Theorem \ref{thm:zeroj}. One can see that the condition $\rho\notin 8\pi \mathbb{N}_{\geq 1}$ plays a crucial role in the proof.

Finally, in Section \ref{Sec 6}, the nondegeneracy and the a priori compactness theorem imply, via
the implicit function theorem, that the number of axisymmetric
solutions is constant on \((8\pi(n-1),8\pi n)\).  The previous results mentioned in Remark \ref{rmk1-3}-(2) give exactly \(n\) solutions for $\rho=8\pi n-4\pi$ or $\rho$ belongs to \eqref{rhoend}, and
the axisymmetry proved above therefore yields exactly \(n\) solutions
throughout the whole interval.

\section{Floquet theory for Hill's equations}\label{sec:floquet}

Let $q\in C(\mathbb R, \mathbb R)$ be a \emph{real-valued} periodic function with minimal period $L>0$. In this section, we first briefly review the basic Floquet theory (see e.g. \cite[Chapter 2]{MagnusWinkler} or \cite[Chapter~1]{Eastham}) for Hill's equation
\begin{equation}\label{eq:Hill}
 -Y''(t)+q(t)Y(t)=EY(t),\quad t\in\mathbb R,
\end{equation}
where $E\in\mathbb C$ is a parameter.

Let $Y_{1}(t)$ and $Y_{2}(t)$ be any two linearly independent solutions of
equation \eqref{eq:Hill}. Then so do $Y_{1}(t+L)$ and $Y_{2}(t+L)$,
hence there exists a monodromy matrix $M(E)\in\mathrm{SL}(2,\mathbb{C})$ such that
\begin{equation}\label{eq:monodromyH}
\begin{pmatrix}Y_{1}(t+L)\\Y_{2}(t+L)\end{pmatrix}=M(E)\begin{pmatrix}Y_{1}(t)\\Y_{2}(t)\end{pmatrix}.
\end{equation}
Here $\det M(E)=1$ follows from the simple fact that the Wronskian $Y_1Y_2'-Y_2Y_1'$ is a constant independent of $t$.

A solution of Hill's equation \eqref{eq:Hill} is called \textit{a Floquet
solution} if it is a eigenfunction of the monodromy matrix $M(E)$.
Define the \emph{Hill's discriminant} $\Delta_q (E)$ by
\begin{equation}
\label{trace}\Delta_q(E):=\text{tr}M(E).
\end{equation}
Since the monodromy matrices of different bases of solutions are conjugate to each other,
the $\Delta_q(E)$ is an invariant of \eqref{eq:Hill}, i.e., does not depend on the choice of the basis of solutions. This $\Delta_q(E)$ is a holomorphic function of $E\in\mathbb{C}$.
It is easy to see that \eqref{eq:Hill} has a periodic solution with period $L$ if and only if $\Delta_q(E)=2$ (i.e., eigenvalues are $1,1$), and such $E$ is called a periodic eigenvalue; \eqref{eq:Hill} has an antiperiodic solution with period $L$ (i.e., $Y(t+L)=-Y(t)$) if and only if $\Delta_q(E)=-2$ (i.e., eigenvalues are $-1,-1$), and such $E$ is called an antiperiodic eigenvalue.
We collect the following result for later usage.

\begin{theorem}\label{thm: Hill}\
\begin{itemize}
\item[$(i)$] If $E\in\mathbb{R}$, then $\Delta_q(E)\in\mathbb{R}$. If $E\in\mathbb{C}\setminus \mathbb{R}$, then $\Delta_q(E)\notin [-2,2]$.
\item[$(ii)$] There are infinitely many periodic eigenvalues $\{E_0\} \cup\{E_{4n-1}, E_{4n}\}^{+\infty}_{n=1}$ and
infinitely many antiperiodic eigenvalues $\{E_{4n-3}, E_{4n-2}\}_{n=1}^{+\infty}$ that are all
real and can be ordered such that
\[
E_0 < \cdots < E_{2n-1} \leq E_{2n} < E_{2n+1} \leq E_{2n+2}<\cdots. \]
\item[$(iii)$] The spectrum $\sigma$ of the linear operator $-\frac{d^2}{dt^2}+q(t)$ in $L^2(\mathbb{R}, \mathbb{C})$ is given by
\[\sigma=\Delta_q^{-1}([-2,2])=\bigcup_{n=0}^{+\infty}[E_{2n}, E_{2n+1}].\]
\item[$(iv)$] $\Delta_q(E)>2$ for any $E<E_0$.
\item[$(v)$] For $n$ large, 
\begin{equation}\label{eq:Hill-zero-asymptotics}E_{2n-1}, E_{2n}=\frac{n^2\pi^2}{L^2}+O(1).\end{equation}
\end{itemize}
\end{theorem}
Here, $(i)$-$(iv)$ are classical results that can be found in \cite[Chapter 2]{MagnusWinkler} or \cite[Chapter~1]{Eastham}, and $(v)$ can be found in \cite[(3.27)]{GW-Acta} or
\cite[Theorem~2.3]{McLaughlinNabelek}. Remark that  $(i)$-$(iv)$ can not hold if $q(t)$ is not real-valued.

Now we prove the following log-concavity of $4-\Delta_q(E)^2$ whenever $4-\Delta_q(E)^2>0$.

\begin{lemma}[Log-concavity]\label{lem:Hill-concavity}
On every open interval $I\subset\mathbb{R}$ where $\Delta_q(E)\in (-2,2)$,
\begin{equation}\label{eq:Hill-concavity}
 \frac{\dd^2}{\dd E^2}\log\bigl(4-\Delta_q(E)^2\bigr)<0.
\end{equation}
\end{lemma}

\begin{proof}
Theorem \ref{thm: Hill} $(ii)$ shows that the zeros of $\Delta_q(E)^2-4$ are precisely $\{E_n\}_{n=0}^{+\infty}$. Consequently, it follows from \eqref{eq:Hill-zero-asymptotics} that the Hadamard factorization of $4-\Delta_q(E)^2$ is given by
\begin{equation}\label{eq:Hadamard}
4-\Delta_q(E)^2=C E^{m_0}\prod_{n\geq 0, E_n\neq 0}
 \left(1-\frac{E}{E_n}\right),
\end{equation}
where $C\neq 0$ is a constant, and $m_0=0$ if $0\notin \{E_n\}_{n=0}^{+\infty}$ (resp. $m_0\in\{1,2\}$ if $0\in \{E_n\}_{n=0}^{+\infty}$). See e.g. \cite[(3.30)]{GW-Acta} for this formula.

Let \(I\subset\mathbb R\) be an open interval where $\Delta_q(E)\in (-2,2)$.  Then \(4-\Delta_q(E)^2>0\) on
\(I\).
Fix a compact subinterval \(I_0\Subset I\).  By
\eqref{eq:Hill-zero-asymptotics}, the series
\[
 \sum_{n\geq 0}\frac1{E-E_n},
 \qquad
 \sum_{n\geq 0}\frac1{(E-E_n)^2}
\]
converge uniformly for \(E\in I_0\). Thus, we can differentiate
the logarithm of \eqref{eq:Hadamard} term by term and obtain
\begin{align*}
 \frac{\dd}{\dd E}\log(4-\Delta_q(E)^2)
 &=\frac{m_0}{E}+\sum_{n\geq 0, E_n\neq 0}\frac1{E-E_n},\\
 \frac{\dd^2}{\dd E^2}\log(4-\Delta_q(E)^2)
 &=-\frac{m_0}{E^2}-\sum_{n\geq 0, E_n\neq 0}
   \frac1{(E-E_n)^2}<0.
\end{align*}
This proves \eqref{eq:Hill-concavity}.
\end{proof}

\section{Monondromy of the Lam\'e equation}\label{Sec 3}

In this section, first we briefly review some basic facts about the monodromy of the Lam\'e equation
\begin{equation}\label{eq:Lame}
 Y''(z)=\bigl[\eta(\eta+1)\wp(z)+B\bigr]Y(z),
\end{equation}
where $\eta>0$ satisifes $\eta\notin \mathbb{N}_{\geq 1}$, and $B\in\mathbb{C}$ is an accessory parameter.

Equation \eqref{eq:Lame} has a regular singularity at $0$ with local exponents $-\eta$ and $\eta+1$.  The monodromy representation of (\ref{eq:Lame}) is a group homomorphism $\rho:\pi_1(E_{\ii b}\setminus\{0\})\rightarrow
\mathrm{SL}(2,\mathbb{C})$ as follows.
Fix a base point $p_0\in E_{\ii b}\setminus\{0\}$
and take $(Y_1(z), Y_2(z))$ to be a basis of local solutions of \eqref{eq:Lame} in a small neighborhood of $p_0$.
For any $\gamma\in\pi_1(E_{ib}\setminus\{0\})$,
we denote by $\gamma^*Y_j(z)$ to be the analytic continuation 
of $Y_j(z)$ along $\gamma$, then there is a matrix $M_{\gamma}(B)\in \mathrm{SL}(2,\mathbb{C})$
such that
\begin{equation}\label{eq:monodromyL}\gamma^*\begin{pmatrix}Y_{1}(z)\\Y_{2}(z)\end{pmatrix}=M_{\gamma}(B)\begin{pmatrix}Y_{1}(z)\\Y_{2}(z)\end{pmatrix},
\end{equation}
where $\det M_{\gamma}(B)=1$ follows from the simple fact that the Wronskian $Y_1Y_2'-Y_2Y_1'$ is a constant independent of $z$. The monodromy group is the image of $\rho:\pi_1(E_{\ii b}\setminus\{0\})\rightarrow
\mathrm{SL}(2,\mathbb{C})$, and is unique up to a common conjugation. In particular, $\tr M_{\gamma}(B)$ is independent of the choice of the basis of local solutions for any $\gamma\in \pi_1(E_{\ii b}\setminus\{0\})$.

Let $\gamma_h\in \pi_1(E_{\ii b}\setminus\{0\})$ (resp. $\gamma_v\in \pi_1(E_{\ii b}\setminus\{0\})$) be the fundamental cycle of $E_{\ii b}$ connecting $p_0$ and $p_0+1$ (resp. connecting $p_0$ and $p_0+\ii b$), and let
$\gamma_{0}\in \pi_1(E_{\ii b}\setminus\{0\})$ be a
simple loop encircling $0$ counterclockwise such that
\[
\gamma_h\gamma_v\gamma_h^{-1}\gamma_v^{-1}=\gamma_{0}\text{ \  \ in
}\pi_1(E_{\ii b}\setminus\{0\})  . \label{II-iv1}%
\]
Consequently, by writing $M_{h}(B)=M_{\gamma_h}(B)$ and $M_{v}(B)=M_{\gamma_v}(B)$ for convenience, we have
\begin{equation}\label{eqMhMnu}[M_h(B), M_v(B)]:=M_h(B)M_{v}(B)M_h(B)^{-1}M_v(B)^{-1}=M_{\gamma_0}(B).\end{equation}
Since the local exponents of \eqref{eq:Lame} at \(0\) are \(-\eta\) and
\(\eta+1\), the two eigenvalues of \(M_{\gamma_0}(B)\) are always
\(e^{-2\pi\ii\eta}\) and
\(e^{2\pi\ii(\eta+1)}=e^{2\pi\ii\eta}\), so
\begin{equation}\label{eq:commutator}
 \tr[M_h(B),M_v(B)]=\tr M_{\gamma_0}(B)=2\cos(2\pi\eta).
\end{equation}
Noting that $M_{\gamma_h\gamma_v}(B)=M_h(B)M_v(B)$, we define
the entire trace functions
\begin{align}\label{eq:traces}
&\mathcal{X}(B):=\tr M_h(B),\qquad \mathcal{Y}(B):=\tr M_v(B),\\
&\mathcal{Z}(B):=\tr(M_{\gamma_h\gamma_v}(B))=\tr(M_h(B)M_v(B)),\nonumber
\end{align}
which are all independent of the choice of the basis of local solutions.

Now we need to use the following classical commutator trace formula.

\begin{lemma}[Commutator trace formula]\label{lem:Fricke}
If $A,C\in\mathrm{SL}(2,\C)$ and
$\mathcal{X}=\tr A$, $\mathcal{Y}=\tr C$, $\mathcal{Z}=\tr(AC)$, then
\begin{equation}\label{eq:Fricke}
 \tr[A,C]=\mathcal{X}^2+\mathcal{Y}^2+\mathcal{Z}^2-\mathcal{X}\mathcal{Y}\mathcal{Z}-2.
\end{equation}
\end{lemma}

\begin{proof} This formula can be seen in \cite[\S2.2.2, (2.2.2.3)]{GoldmanFricke}.
\end{proof}

Combining \eqref{eq:commutator} and \eqref{eq:Fricke}, we obtain the identity
\begin{equation}\label{eq:Fricke-square}
(2\mathcal{Z}(B)-\mathcal{X}(B)\mathcal{Y}(B))^2=(4-\mathcal{X}(B)^2)(4-\mathcal{Y}(B)^2)-16\sin^2(\pi\eta)
\end{equation}
for any $B\in\C$.

To proceed, we use the basic property of the Weierstrass $\wp$ function $\wp(z)=\wp(z;\ii b)$ for a rectangular lattice:
\begin{equation}\label{eq:wp-real}
 \overline{\wp(\bar z)}=\wp(z).
\end{equation}
Define
\begin{equation}\label{eq:zb}
 z_b:=-\frac12-\frac{\ii b}{2},
\end{equation}
and
\begin{align*}&q_h(t):=\eta(\eta+1)\wp(z_b+t),\quad t\in\mathbb R,\\
&q_v(t):=-\eta(\eta+1)\wp(z_b+\ii t),\quad t\in\mathbb R.\end{align*}
Clearly $q_h(t)$ and $q_v(t)$ have no singularities on $\mathbb R$ and hence are smooth on $\mathbb R$.
Since
$$\overline{z_b+t}\equiv z_b+t,\quad\overline{z_b+\ii t}\equiv -(z_b+\ii t)\pmod{\mathbb Z+\mathbb Z\ii b},$$
it follows from \eqref{eq:wp-real}  and $\wp(-z)=\wp(z)$ that both $q_h(t)$ and $q_v(t)$ are \emph{real-valued}. Note also that $q_h(t+1)=q_h(t)$ and $q_v(t+b)=q_v(t)$. Therefore, we can apply those results in Section \ref{sec:floquet} to Hill's equations
\begin{align}
 &-Y''(t)+q_h(t)Y(t)=(-B)Y(t),\qquad t\in\mathbb R,\label{eq:horizontal-Hill}\\
 &-Y''(t)+q_v(t)Y(t)=BY(t),\qquad t\in\R.\label{eq:vertical-Hill}
\end{align}
If $Y(z)$ is a solution of the Lam\'{e} equation \eqref{eq:Lame}, then $Y(z_b+t)$ is a solution of \eqref{eq:horizontal-Hill} and $Y(z_b+\ii t)$ is a solution of \eqref{eq:vertical-Hill} as a function of $t$. Consequently, we see from \eqref{eq:monodromyH}-\eqref{trace}, \eqref{eq:monodromyL} and \eqref{eq:traces} that
\begin{equation}\label{eq:Hill-traces}
 \mathcal{X}(B)=\Delta_{h}(-B),\qquad \mathcal{Y}(B)=\Delta_v(B),
\end{equation}
where $\Delta_h(-B):=\Delta_{q_h}(-B)$ is the Hill's discriminant of \eqref{eq:horizontal-Hill} and $\Delta_v(B):=\Delta_{q_v}(B)$ is the Hill's discriminant of \eqref{eq:vertical-Hill}. In particular, Theorem \ref{thm: Hill} $(i)$ implies
\begin{equation}\label{eq:xy-b}
\mathcal{X}(B)\in\mathbb R,\quad \mathcal{Y}(B)\in\mathbb R,\quad\forall B\in\mathbb R.
\end{equation}
Remark that $\mathcal{Z}(B)$ can not be real-valued for $B\in\mathbb R$ in general.
Define
\begin{align}\label{eq:PCG}
& P=P(B):=(4-\mathcal{X}(B)^2)(4-\mathcal{Y}(B)^2),\\
&G=G(B):=2\mathcal{Z}(B)-\mathcal{X}(B)\mathcal{Y}(B),\nonumber\\
& C_\eta:=16\sin^2(\pi\eta)>0,\quad\text{thanks to }\eta\notin\mathbb{Z}.\nonumber
\end{align}
Then \eqref{eq:Fricke-square} is just
\begin{equation}\label{eq:G-square}
 G(B)^2=P(B)-C_\eta,
\end{equation}
and $G(B)^2$ is real-valued for $B\in\mathbb{R}$ although $G(B)$ might not.
We need to establish the following important lemma.

\begin{lemma}\label{lem:no-positive}
Let $I$ be a connected component of
\begin{equation}\label{eq:common-stability}
 \{B\in\R\;:\; \mathcal{X}(B)\in (-2,2),\; \mathcal{Y}(B)\in (-2,2)\}.
\end{equation}
Then $0<P(B)\le C_\eta$ for any $B\in I$. Furthermore, if $P(B_0)=C_\eta$ for some $B_0\in I$, then $B_0$ is the
unique equality point in $I$ and
\begin{equation}\label{eq:strict-contact}
 P'(B_0)=0,\qquad P''(B_0)<0.
\end{equation}
\end{lemma}

\begin{proof} 
By Lemma~\ref{lem:Hill-concavity} and \eqref{eq:Hill-traces}, we get that for any $B\in I$,
\begin{equation}\label{eq:logP}
 (\log P(B))''=
 \bigl[\log(4-\Delta_h^2)\bigr]''(-B)
 +\bigl[\log(4-\Delta_v^2)\bigr]''(B)<0.
\end{equation}
Thus \(f(B):=\log P(B)\) is strictly concave on \(I\).

By Theorem \ref{thm: Hill} $(iv)$ and \eqref{eq:Hill-traces}, there is a large $\widetilde B>0$ such that
$$\mathcal{X}(B)=\Delta_{h}(-B)>2,\quad\forall B>\widetilde B,$$
$$\mathcal{Y}(B)=\Delta_{v}(B)>2,\quad\forall B<-\widetilde B.$$
This implies that as a connected component of  \eqref{eq:common-stability},
 \(I=(\alpha,\beta)\) is a bounded open interval.  By the maximality of the component
and continuity of \(\mathcal{X}(B)\) and \(\mathcal{Y}(B)\), at each endpoint of $I$ at least one of the
equalities \(\mathcal{X}=\pm 2\) or \(\mathcal{Y}=\pm 2\) holds, so
\begin{equation}\label{eq:P-endpoints}
 P(\alpha)
 =P(\beta)=0.
\end{equation}

Suppose \(P(B)>C_\eta>0\) for some $B\in I$.  By \eqref{eq:P-endpoints}, there is a connected component
\((a,d)\) of \(\{P>C_\eta\}\) containing this $B$ such that
\begin{equation}\label{eq:papd}
 [a,d]\Subset I,\qquad P(a)=P(d)=C_\eta.
\end{equation}
Take any \(\xi\in(a,d)\), then the strict concavity implies
\begin{align*}
 f'(a)\ge \frac{f(\xi)-f(a)}{\xi-a}>0,\quad
 f'(d)\le \frac{f(d)-f(\xi)}{d-\xi}<0.
\end{align*}
  Since \(P>0\) and $P'=Pf'$ on \(I\), so
$$P'(a)>0,\qquad P'(d)<0.$$

On the other hand, \eqref{eq:G-square} and \eqref{eq:papd} imply \(G(a)^2=G(d)^2=0\),
i.e., \(G(a)=G(d)=0\).  Differentiating  \eqref{eq:G-square}  gives
\[
 P'(a)=2G(a)G'(a)=0,\qquad
 P'(d)=2G(d)G'(d)=0,
\]
a contradiction.  Thus \(P(B)\le C_\eta\) for any \(B\in I\).

Suppose $P(B_0)=C_\eta>0$ for some $B_0\in I$.  Then \eqref{eq:G-square} implies
\(G(B_0)=0\) and so
\(P'(B_0)=2G(B_0)G'(B_0)=0\).  Together with $P'=Pf'$, we get \(f'(B_0)=0\).
Because \(f''(B)<0\) on $I$, the function \(f'\) is strictly decreasing, so it
can vanish at most once.  This proves that $B_0$ is the unique point in $I$ such that $P=C_\eta$ holds.
Finally, differentiating \(P=\exp f\) twice gives
\[
 P''=P\bigl(f''+(f')^2\bigr).
\]
Since \(f'(B_0)=0\), we get
$
P''(B_0)
 =P(B_0)f''(B_0)<0.
$
This completes the proof.
\end{proof}

\section{Axisymmetry and Lam\'e transversality}\label{Sec 4}

Let \(\rho=8\pi\eta\notin8\pi\N_{\geq 1}\) and $u(z)$ be a solution of the singular Liouville equation \eqref{eq:mfe}. 
First, we briefly review its basic connection with the Lam\'{e} equation \cite[Section 3]{ChaiLinWang}.
Define
\begin{equation}\label{eq:Tdef}
 T(z)=T[u](z):=u_{zz}(z)-\frac12u_z(z)^2.
\end{equation}
Using $u_{z\bar z}=-e^u/4$, we get
\[
 T_{\bar z}=u_{zz\bar z}-u_zu_{z\bar z}=-\frac14e^uu_z+\frac14e^uu_z=0,\quad z\in E_{\ii b}\setminus\{0\},
\]
so $T$ is holomorphic in $E_{\ii b}\setminus\{0\}$.  By \eqref{ulog},  we see that
$$T(z)=\frac{-2\eta(\eta+1)}{z^2}+O(z^{-1}),\quad\text{near }z=0,$$
so $T(z)$ is an elliptic function with a unique (double) pole at $0$.
Furthermore, the residue of $T(z)$ at $0$ is zero because an
elliptic function has total residue zero.  Hence there is $B_0=B_0(u)\in \mathbb{C}$ such that
\begin{equation}\label{eq:TLame}
 T(z)=-2\bigl[\eta(\eta+1)\wp(z)+B_0\bigr].
\end{equation}

On the other hand,
the Liouville theorem says that there
is a local meromorphic function $f(z)$ away from the singularity $0$ such that%
\begin{equation}
u(z)=\log \frac{8|f^{\prime}(z)|^{2}}{(1+|f(z)|^{2})^{2}}. \label{502}
\end{equation}
We remark that the classical Liouville theorem holds only for the case when
the domain is simply connected and the equation does not include any
singularity. For our present case, see \cite[Section 1]{ChaiLinWang} for a proof.

This $f(z)$ is called a developing map of $u(z)$ or the spherical metric $g=\frac12 e^{u}|dz|^2$. By differentiating (\ref{502}), we
have
\begin{equation}
\mathcal{S}(f):=\frac{f^{\prime\prime\prime}
}{f^{\prime}}-\frac{3}{2}\left(  \frac{f^{\prime \prime}%
}{f^{\prime}}\right)  ^{2}=u_{zz}-\frac{1}{2}u_{z}^{2}=T. \label{new22}%
\end{equation}
Conventionally, $\mathcal{S}(f)$ is called the Schwarzian derivative
of $f(z)$. 
Since $\mathcal{S}(f)=T(z)=-2[\eta(\eta+1)\wp(z)+B_0]$, a classical result in complex analysis says that there exists a local basis of solutions $Y_1(z), Y_2(z)$ (in a small neighborhood of the base point $p_0$) of the corresponding Lam\'{e} equation
\begin{equation}\label{eq:Lame0}
 Y''(z)=\bigl[\eta(\eta+1)\wp(z)+B_0\bigr]Y(z),
\end{equation}
 such that
\begin{equation}\label{fy}f(z)=\frac{Y_1(z)}{Y_2(z)}.\end{equation}
For any $\gamma\in\pi_1(E_{\ii b}\setminus\{0\})$, the analytic
continuation gives
\[
\gamma^*
\begin{pmatrix}
Y_1\\ Y_2
\end{pmatrix}
=
M_\gamma(B_0)
\begin{pmatrix}
Y_1\\ Y_2
\end{pmatrix},
\qquad
M_\gamma(B_0)\in\mathrm{SL}(2,\mathbb C).
\]
Hence, if
\[
M_\gamma(B_0)=
\begin{pmatrix}
a_\gamma&b_\gamma\\
c_\gamma&d_\gamma
\end{pmatrix},
\]
then
\[
\gamma^*f=f\circ\gamma
=\frac{a_\gamma f+b_\gamma}{c_\gamma f+d_\gamma}.
\]
The matrices $M_\gamma(B_0)$ and $-M_\gamma(B_0)$ induce the same
M\"obius transformation. Therefore the action on $f$ depends only on
the projective class of $M_\gamma(B_0)$. Thus the natural projection
\[
\pi:\mathrm{SL}(2,\C)\longrightarrow\mathrm{PSL}(2,\C)
\]
defines the projective monodromy representation
\[
\varrho(\gamma):=\pi\bigl(M_\gamma(B_0)\bigr),
\qquad
\gamma^*f=\varrho(\gamma)\circ f=\frac{a_\gamma f+b_\gamma}{c_\gamma f+d_\gamma}.
\]

The Wronskian
$W:=Y_{1}(z)Y_2'(z)-Y_1'(z)Y_2(z)$ is a nonzero constant. By replacing $(Y_1(z), Y_2(z))$ with $(Y_1(z)/W^{\frac12}, Y_2(z)/W^{\frac12})$ if necessary, we may always assume $W\equiv 1$. By inserting (\ref{fy}) into (\ref{502}), a direct computation leads to
\begin{equation}\label{eq:wy1y2} e^{-\frac{1}{2}u(z)}=\frac{1}{2\sqrt{2}}(|Y_1(z)|^2+|Y_2(z)|^2).\end{equation}
Since $u(z)$ is single-valued and doubly periodic, we immediately see that for any $\gamma\in\pi_1(E_{\ii b}\setminus\{0\})$, the monodromy matrix $M_{\gamma}(B_0)$ with respect to $(Y_1(z), Y_2(z))$ satisfies
\begin{equation}\label{eq:y1y2}|Y_1(z)|^2+|Y_2(z)|^2=(\overline{Y_1(z)},\overline{Y_2(z)})\overline{M_{\gamma}(B_0)}^{\mathsf T}M_{\gamma}(B_0)\begin{pmatrix}Y_1(z)\\Y_2(z)\end{pmatrix},\end{equation}
so $M_{\gamma}(B_0)\in \mathrm{SU}(2)$ is a unitary matrix. Consequently, the monodromy group of \eqref{eq:Lame0} with respect to $(Y_1(z), Y_2(z))$ is a subgroup of
 $\mathrm{SU}(2)$, namely the monodromy is unitarizable. Since the trace of any unitary matrix belongs to $[-2,2]$, we conclude from \eqref{eq:traces} that
\begin{equation}\label{eq:traces0}
\mathcal{X}(B_0)\in [-2,2],\quad \mathcal{Y}(B_0)\in [-2,2],\quad
\mathcal{Z}(B_0)\in [-2,2].
\end{equation}
In particular, the projective monodromy group of the developing map $f(z)$ is contained in $\mathrm{PSU}(2)$.

Now we prove the symmetry of $u(z)$.

\begin{lemma}\label{symmetric lem}
\label{lem:reflection-rigidity} We have $B_0\in\mathbb R$ and
\[
u(z)=u(-z)= u(\bar z),
\]
namely $u(z)$ is axisymmetric.
\end{lemma}

\begin{proof}
Since \eqref{eq:Hill-traces} and \eqref{eq:traces0} imply $\Delta_h(-B_0)=\mathcal{X}(B_0)\in [-2,2]$, it follows from Theorem \ref{thm: Hill} $(i)$ that 
\(B_0\in\R\).
On the other hand, thanks to $\eta\notin\mathbb{Z}$, it was proved in \cite[Lemma~4.3]{LinWang} that
$u(z)$ is even, i.e., $u(-z)=u(z).$

Thus, it suffices to prove $u(z)=u(\bar z)$. Since $E_{\ii b}$ is a rectangular torus,
we see that
\[
 \tilde{u}(z):=u(\bar z)
\]
is also a solution of \eqref{eq:mfe}.  Furthermore, it follows from \eqref{eq:wp-real} and \(B_0\in\R\) that
\begin{align}
 T[\tilde u](z)
 &=\overline{T[u](\bar z)}
 =-2\bigl[\eta(\eta+1)\overline{\wp(\bar z)}+\overline {B_0}\bigr]\nonumber\\
  &=-2\bigl[\eta(\eta+1)\wp(z)+B_0\bigr]
  =T[u](z),\label{eq:reflected-T}
\end{align}
namely $u(z)$ and $\tilde u(z)$ corresponds to the same $B_0\in\mathbb R$.

By repeating the arguments \eqref{502}-\eqref{eq:y1y2}, there is a local basis of solutions $\tilde Y_1(z), \tilde Y_2(z)$ of the same Lam\'{e} equation
\eqref{eq:Lame0} such that $\tilde Y_{1}(z)\tilde Y_2'(z)-\tilde Y_1'(z)\tilde Y_2(z)\equiv1$,
$\tilde f(z):=\tilde Y_1(z)/\tilde Y_2(z)$ is a developing map of $\tilde u(z)$,
\begin{equation}\label{eq:0wy1y2} e^{-\frac{1}{2}\tilde u(z)}=\frac{1}{2\sqrt{2}}(|\tilde Y_1(z)|^2+|\tilde Y_2(z)|^2),\end{equation}
 and the monodromy matrices $\tilde M_\gamma (B_0)$ with respect to 
$(\tilde Y_1(z), \tilde Y_2(z))$ are also unitary matrices.

Clearly there is an invertible matrix $A$ such that
\begin{equation}\label{eqy1y2w}\begin{pmatrix}\tilde Y_1(z)\\ \tilde Y_2(z)\end{pmatrix}=A\begin{pmatrix}Y_1(z)\\Y_2(z)\end{pmatrix}\quad\text{and}\quad \det A=1.\end{equation}
Then
\begin{equation}\label{eqmb0}\tilde M_\gamma (B_0)=AM_\gamma (B_0) A^{-1},\quad\forall \gamma\in \pi_1(E_{\ii b}\setminus\{0\}).\end{equation}
Set
\[
 H:=\bar{A}^{\mathsf T}A.
\]
Using $\tilde M_\gamma (B_0), M_\gamma (B_0)\in \mathrm{SU}(2)$, we easily see from \eqref{eqmb0} that
\begin{equation}\label{eq:H-invariant}
 HM_\gamma (B_0)=M_\gamma (B_0) H, \quad\forall \gamma\in \pi_1(E_{\ii b}\setminus\{0\}).
\end{equation}
The matrix \(H\) is positive definite Hermitian and so can be diagonalized as
$$A_1HA_1^{-1}=\begin{pmatrix}\lambda_1&\\ &\lambda_2\end{pmatrix}$$
for some $\lambda_j>0$ and invertible matrix $A_1$. Suppose $\lambda_1\neq \lambda_2$, then  \eqref{eq:H-invariant} implies
$$\begin{pmatrix}\lambda_1&\\ &\lambda_2\end{pmatrix}A_1M_\gamma (B_0)A_1^{-1}=A_1M_\gamma (B_0)A_1^{-1}\begin{pmatrix}\lambda_1&\\ &\lambda_2\end{pmatrix},$$
so $A_1M_\gamma (B_0)A_1^{-1}$ is diagonal for any $\gamma\in \pi_1(E_{\ii b}\setminus\{0\})$. But this together with \eqref{eqMhMnu} implies that
$M_{\gamma_0}(B_0)$ is the identity matrix $I_2=\mathrm{diag}(1,1)$, a contradiction with $\mathrm{tr}M_{\gamma_0}(B_0)=2\cos(2\pi\eta)\neq 2$.

Thus, \(\bar{A}^{\mathsf T}A=H=\lambda I_2\) for some \(\lambda>0\), and $\det A=1$ implies $\lambda=1$, i.e.,
\(\bar{A}^{\mathsf T}A=I_2\).  From here and \eqref{eqy1y2w}, we obtain
$$|\tilde Y_1(z)|^2+|\tilde Y_2(z)|^2=|Y_1(z)|^2+|Y_2(z)|^2.$$
This proves $u(z)=\tilde u(z)=u(\bar z)$ by using \eqref{eq:wy1y2} and \eqref{eq:0wy1y2}.
\end{proof}

\begin{corollary}\label{coro:11}
Let $n\in\mathbb N_{\ge 1}$ and $\rho=8\pi\eta$ with $\eta\in (n-1,n)$. Then there is a one-to-one correspondence between solutions of \eqref{eq:mfe} and those $B$'s such that the monodromy of \eqref{eq:Lame} is unitarizable.
\end{corollary}

\begin{proof}
Take $B$ such that the monodromy of \eqref{eq:Lame} is unitarizable. Then there is a local basis of solutions $Y_1(z), Y_2(z)$ such that the monodromy matrices $M_\gamma(B)\in \mathrm{SU}(2)$. Define $f:=Y_1/Y_2$ and define $u$ via this $f$ by the Liouville formula \eqref{502}. Then it is easy to prove that $u(z)$ is a solution of \eqref{eq:mfe} corresponding to this $B$; see e.g. \cite[Section 3]{ChaiLinWang} for details. Finally, the proof of Lemma \ref{symmetric lem} implies that the solution of \eqref{eq:mfe} corresponding to this $B$ is unique. This completes the proof.
\end{proof}

The main result of this section is the following Lam\'e transversality, which plays a crucial role in the proof of the nondegeneracy in the next section.

\begin{theorem}\label{thm:transversality}
We have
\begin{equation}\label{eq:mixed-transverse}
 \frac{\dd}{\dd B}\mathrm{Im}\tr(M_h(B)M_v(B))\bigg|_{B=B_0}\ne0.
\end{equation}
\end{theorem}

\begin{proof} Recall \eqref{eq:traces0} that
\[
\mathcal{X}(B_0)\in [-2,2],\quad \mathcal{Y}(B_0)\in [-2,2],\quad
\mathcal{Z}(B_0)\in [-2,2].
\]
Note that for a unitary matrix, its trace is $\pm 2$ if and only if it is $\pm I_2$.
If \(\mathcal{X}(B_0)=\pm2\), then
\(M_h(B_0)=\pm I_2\), so the commutator $[M_h(B_0),M_v(B_0)]$ is \(I_2\),  a contradiction with
\[
 \tr[M_h(B_0),M_v(B_0)]=2\cos(2\pi\eta)\ne2
\]
because \(\eta\notin\mathbb Z\).  The same argument applies to
\(\mathcal{Y}(B_0)\).  Therefore,
\[
\mathcal{X}(B_0)\in (-2,2),\quad \mathcal{Y}(B_0)\in (-2,2),
\]
namely \(B_0\) belongs to a connected component \(I\) of
\eqref{eq:common-stability}.  Since
$
 \mathcal{X}(B_0), \mathcal{Y}(B_0)$, $\mathcal{Z}(B_0)
$
are real, we have \(G(B_0)\in\R\), so \(G(B_0)^2\ge0\).  On the other
hand, Lemma~\ref{lem:no-positive} gives
\(G(B_0)^2=P(B_0)-C_\eta\le0\).  This implies
\begin{equation}\label{eq:Gzero}
 G(B_0)=0,\qquad P(B_0)=C_\eta.
\end{equation}
Differentiating \eqref{eq:G-square} leads to
\begin{align}
 2(G')^2+2GG''=P''.\label{eq:G-square-second}
\end{align}
Using \eqref{eq:Gzero}, \eqref{eq:G-square-second} and \eqref{eq:strict-contact}, we obtain
\begin{equation}\label{eq:Gprime}
 2G'(B_0)^2=P''(B_0)<0.
\end{equation}
Thus
\[
 G'(B_0)=\ii a_1,\qquad\text{for some }\; a_1\in\mathbb{R}\setminus\{0\}.
\]

Recall \eqref{eq:xy-b} that
 \(\mathcal{X}(B), \mathcal{Y}(B)\in\R\) for $B\in I$, so \(\mathcal{X}'(B),\mathcal{Y}'(B)\in \mathbb R\) for $B\in I$. 
By \(G(B)=2\mathcal{Z}(B)-\mathcal{X}(B)\mathcal{Y}(B)\), we have
\[G'(B_0)-2\mathcal{Z}'(B_0)=-\mathcal{X}'(B_0)\mathcal{Y}(B_0)-\mathcal{X}(B_0)\mathcal{Y}'(B_0)\in\mathbb R,\]
which finally implies
\[
\frac{\dd}{\dd B}\mathrm{Im} \mathcal{Z}(B)\bigg|_{B=B_0}= \frac12\mathrm{Im} G'(B_0)=\frac12 a_1\neq 0.
\]
The proof is complete.
\end{proof}

\section{Infinitesimal monodromy of the exact Lam\'e path}
\label{sec:axisymmetric-nondegeneracy}

In this section, we prove the nondegeneracy result Theorem \ref{thm:zeroj}. It is equivalent to showing that every axisymmetric Jacobi field $\phi$ of \eqref{eq:mfe} is trivial. Set
\[
X_{\rm ax}:=\{f:f(z)=f(-z)=f(\bar z)\}.
\]
Let \(u\in X_{\rm ax}\) be a solution, and suppose that \(\phi\) is a Jacobi field satisfying
\begin{equation}\label{eq:Jacobi-3}
0\neq \phi\in X_{\rm ax},
\qquad
(\Delta+e^u)\phi=0.
\end{equation}
Our goal is to obtain a contradiction with Theorem \ref{thm:transversality}.

Define the Jacobi differential by
\begin{equation}\label{eq:Qdef}
Q_\phi:=\phi_{zz}-u_z\phi_z .
\end{equation}
The first simple observation is that \(Q_\phi\) must be a real number.

\begin{proposition}\label{prop:Qconstant}
There exists $c\in\R$ such that 
\begin{equation}\label{eq:Qc}
 Q_\phi=c.
\end{equation}
\end{proposition}

\begin{proof}
Using the equations for $u$ and $\phi$, a direct computation gives
\begin{align*}
 (Q_\phi)_{\bar z}
 &=\phi_{zz\bar z}-u_{z\bar z}\phi_z-u_z\phi_{z\bar z}\\
 &=-\tfrac14e^u(u_z\phi+\phi_z)
   +\tfrac14e^u\phi_z+\tfrac14u_ze^u\phi=0.
\end{align*}
Thus $Q_\phi$ is holomorphic on the punctured torus. Since $\phi\in C^2(E_{\ii b})$ and $u(z)=\frac{\rho}{2\pi}\log|z|+O(1)$ near $0$, the function $Q_{\phi}$ has only one possible singularity, which is of order one at $0$. Therefore, $Q_{\phi}$ must be constant. Moreover, using
$
\phi_z(\bar z)=\overline{\phi_z(z)}, u_z(\bar z)=\overline{u_z(z)},
$
we obtain $c=\bar c$, so this constant is real.
\end{proof}

\begin{proposition}\label{prop:c-nonzero}
The constant $c$ in \eqref{eq:Qc} is nonzero.
\end{proposition}

\begin{proof}
Assume for contradiction that $c=0$, and set
$\psi:=e^{-u}\phi_{\bar z}$. Since $u$ and $\phi$ are real, the
complex conjugate of \eqref{eq:Qdef} gives
$$\phi_{\bar z\bar z}-u_{\bar z}\phi_{\bar z}=0,$$ and hence
$\psi_{\bar z}=0$. Thus $\psi$ is a single-valued holomorphic function
on $E_{\ii b}\setminus\{0\}$. Moreover, $\psi\not\equiv0$. Indeed, otherwise
$\phi_{\bar z}=0$, so the real-valued function $\phi$ is a constant, and
the Jacobi equation \eqref{eq:Jacobi-3} then gives $\phi=0$, contrary to \eqref{eq:Jacobi-3}.

Since $4\phi_{z\bar z}+e^u\phi=0$ and
$\phi_{\bar z}=e^u\psi$, we have
\begin{align}\label{Sec 5 equ-1}
\phi=-4e^{-u}\phi_{z\bar z}=-4(\psi_z+u_z\psi). 
\end{align}
Writing
\begin{equation}\label{eqTT}
T=u_{zz}-\frac12u_z^2=-2[\eta(\eta+1)\wp(z)+B_0],
\end{equation}
and differentiating \eqref{Sec 5 equ-1}, we obtain
\[
0=\phi_{zz}-u_z\phi_z
=-4\bigl(\psi_{zzz}+2T\psi_z+T_z\psi\bigr).
\]
Thus $\psi$ satisfies the third-order equation
\begin{equation}\label{eq:symmetric-square}
\psi'''+2T\psi'+T'\psi=0.
\end{equation}

Let $Y_1,Y_2$ be the  Wronskian-one fundamental system of the Lam\'e equation
\[
Y''=-\frac12TY=[\eta(\eta+1)\wp(z)+B_0]Y
\]
given by \eqref{fy}-\eqref{eq:wy1y2} in Section \ref{Sec 4}. A direct
calculation shows that
\[
Y_1^2,\qquad Y_1Y_2,\qquad Y_2^2
\]
form a  fundamental system of \eqref{eq:symmetric-square}. Therefore, by $\psi\not\equiv 0$,
\begin{equation}\label{eq:syms}
\psi=\mathbf Y^{\mathsf T}H\mathbf Y,\qquad
\mathbf Y:=
\begin{pmatrix}
Y_1\\
Y_2
\end{pmatrix},
\end{equation}
for a unique nonzero symmetric matrix
$H\in\operatorname{M}_{2\times2}(\mathbb C)$. 

With the monodromy convention fixed in Section \ref{Sec 4}, the analytic
		continuation along $\gamma$ gives
		$\gamma^*\mathbf Y=M_\gamma\mathbf Y$, where we write $M_\gamma=M_\gamma(B_0)$ for convenience. Since $\psi$ is single-valued, then $\psi\circ\gamma=\psi$,
		and the uniqueness of the representation \eqref{eq:syms} gives
		\begin{equation}\label{eq:H-invariance}
			M_\gamma^{\mathsf T}HM_\gamma=H
			\qquad\text{for every }\gamma\in\pi_1(E_{\ii b}\setminus\{0\}).
		\end{equation}
		Let
		\[
		J:=
		\begin{pmatrix}
			0&1\\
			-1&0
		\end{pmatrix},
		\qquad
		A:=J^{-1}H.
		\]
		For every $M\in\mathrm{SL}(2,\mathbb C)$ one has
		$M^{\mathsf T}JM=J$. Combining this identity with
		\eqref{eq:H-invariance} gives
		\[
		M_\gamma^{-1}AM_\gamma=A
		\qquad\text{for every }\gamma\in\pi_1(E_{\ii b}\setminus\{0\}).
		\]
		Since $H$ is symmetric, $\operatorname{tr}A=0$. Moreover $A\ne0$
		because $H\ne0$, and hence $A$ is not a scalar matrix, namely, $A\not= cI_2$ for any constant $c$.
		
		Since $A$ is a nonscalar $2\times2$ matrix, it is nonderogatory.
		Hence its centralizer is
		\begin{align}\label{Sec 5 center}
		C(A)=\{B: AB=BA\}
		=\operatorname{span}_{\mathbb C}\{I_2,A\},  
		\end{align}
		and is therefore commutative; see, for example,
		\cite[Corollary~4.4.18]{HornJohnsonTopics}. For the convenience of the reader, we provide the details of \eqref{Sec 5 center}.
		Indeed, choose \(P\in \mathrm{GL}(2,\mathbb C)\) such that \(\Lambda:=P^{-1}AP\) is in the Jordan normal form. For \(B\in C(A)\), set \(\widetilde B:=P^{-1}BP\). Then \(\Lambda\widetilde B=\widetilde B\Lambda\). There are two cases to consider:

 \textbf{Case 1:} If \(A\) has two distinct eigenvalues. Then \(\Lambda=\begin{pmatrix}\lambda_1&\\ &\lambda_2\end{pmatrix}\) with \(\lambda_1\neq\lambda_2\). Writing
\(\widetilde B=\begin{pmatrix}a&b\\ c&d\end{pmatrix}\), the relation
\(\Lambda\widetilde B=\widetilde B\Lambda\) gives
\((\lambda_1-\lambda_2)b=(\lambda_1-\lambda_2)c=0\), so \(b=c=0\).
Thus \(\widetilde B=\begin{pmatrix}a&\\ &d\end{pmatrix}=\alpha I_2+\beta \Lambda\) for suitable
\(\alpha,\beta\in\mathbb C\).

\textbf{Case 2:}  If \(A\) has only one eigenvalue \(\lambda\). Then, since \(A\) is non-scalar,
\[
\Lambda=\begin{pmatrix}\lambda&1\\0&\lambda\end{pmatrix}.
\]
Writing again \(\widetilde B=\begin{pmatrix}a&b\\ c&d\end{pmatrix}\), the
identity \(\Lambda\widetilde B=\widetilde B\Lambda\) yields \(c=0\) and \(d=a\). Hence
\(\widetilde B=\begin{pmatrix}a&b\\0&a\end{pmatrix}
=(a-\lambda b)I_2+b\Lambda\).

Therefore, in either case, \(\widetilde B=\alpha I_2+\beta \Lambda\), and conjugating
back gives
\[
B=\alpha I_2+\beta A.
\]
Hence \(C(A)=\operatorname{span}_{\mathbb C}\{I_2,A\}\), which is commutative.
Consequently, since every monodromy matrix \(M_\gamma\) commutes with \(A\),
all the \(M_\gamma\)'s commute with one another. In particular,
\[
[M_h,M_v]=I_2,
\qquad
\operatorname{tr}[M_h,M_v]=2,
\]
		again a contradiction with 
		\[
		\operatorname{tr}[M_h,M_v]=2\cos(2\pi\eta)\neq 2,
		\]
		because $\eta\notin\mathbb Z$. This contradiction proves that $c\ne0$.
\end{proof}

\medskip
\noindent\textbf{The exact Lam\'e path and an infinitesimal deformation argument.}
The real constant $Q_\phi=c$ describes the first-order variation of the
projective connection generated by the formal perturbation $u+t\phi$.
Indeed, recalling \eqref{eqTT}, we have
\[
T[u+t\phi]
=T+tc-\frac{t^2}{2}\phi_z^2,
\qquad
\text{where }\; T[v]:=v_{zz}-\frac12v_z^2.
\]
However, $u+t\phi$ does not solve the Liouville equation for $t\ne0$.
We therefore introduce the exact Lam\'e path
\begin{equation}\label{eq:exact-q-path-preview}
q_t:=T+tc
=-2\bigl[\eta(\eta+1)\wp+B(t)\bigr],
\quad\text{where }\;
B(t):=B_0-\frac c2t,
\end{equation}
so that
\[
\dot B(0)=-\frac c2\ne0.
\]
Thus $q_t$ agrees with the projective connection $T[u+t\phi]$ of $u+t\phi$ to first
order at $t=0$, while it remains an exact Lam\'e projective connection
for every $t$. The next theorem compares the infinitesimal deformation
generated by $\phi$ with the monodromy of this exact path $q_t$.

More precisely, we want to realize the Jacobi deformation $u+t\phi$ by an infinitesimal
deformation $f_t$ of the developing map $f$.
Denote by $M:=E_{\ii b}\setminus\{0\}$ for convenience, and let $\widetilde M$ be the universal cover of $M$. As usual, we use the same notation
for the lifts of $u$ and $\phi$ to $\widetilde M$.  
Let \(\widetilde p_0\in\widetilde M\) be a fixed lift of the base point \(p_0\in M\) used in Section \ref{Sec 3}. For simplicity, we still denote it by \(p_0\). Choose the two fixed lifted period paths from $p_0$ to $\gamma_hp_0$
and $\gamma_v p_0$ so that their union $\Gamma$ is a finite tree.
Let $g=\frac12 e^u|dz|^2$, and let $$F_0: (\widetilde M, \widetilde g) \to (\mathbb S^2, g_{\mathbb S^2})$$ be an orientation-preserving developing map, which is a local isometry with respect to the
round metric $g_{\mathbb S^2}$ on \(\mathbb S^2\).
Since
$F_0(\Gamma)$ is a finite union of smooth arcs, we may choose
$a_*\in\mathbb S^2\setminus F_0(\Gamma)$. Let
$$\sigma_*:\mathbb S^2\setminus\{a_*\}\to\mathbb C$$ be an
orientation-preserving stereographic coordinate. After taking a
sufficiently thin simply connected neighborhood $\Omega$ of $\Gamma$,
we may assume that $\overline\Omega\Subset\widetilde M$ and
$F_0(\overline\Omega)\cap\{a_*\}=\varnothing$. Set
$f:=\sigma_*\circ F_0$ on $\Omega$. Then $f$ is holomorphic and locally
univalent, and is a developing map of $u$ mentioned in \eqref{502}, i.e.,
\begin{equation}\label{eq:direct-developing-formula}
e^u=\frac{8|f_z|^2}{(1+|f|^2)^2},
\qquad
\mathcal S(f)=T,
\end{equation}
where $\mathcal S(f)$ is the Schwarzian derivative defined in \eqref{new22}.

We now identify the infinitesimal metric variation generated by a
holomorphic deformation of $f$. For a holomorphic function $v$ on
$\Omega$, define
\begin{equation}\label{eq:direct-Vf}
\mathcal V_f[v]
:=
2\operatorname{Re}\left(
\frac{v_z}{f_z}
-\frac{2\overline f\,v}{1+|f|^2}
\right).
\end{equation}
Indeed, for real $s$, set $f_s=f+sv$ and
\[
u_s:=\log\frac{8|(f_s)_z|^2}{(1+|f_s|^2)^2}.
\]
Since $f_z\ne0$, the map $f_s$ remains locally univalent on compact
subsets of $\Omega$ for all sufficiently small $|s|$. Differentiating
at $s=0$ gives
\[
\left.\frac{\partial u_s}{\partial s}\right|_{s=0}
=
2\operatorname{Re}\left(
\frac{v_z}{f_z}
-\frac{2\overline f\,v}{1+|f|^2}
\right)
=\mathcal V_f[v].
\]
Thus $\mathcal V_f[v]$ is precisely the first variation of the
Liouville conformal factor induced by the deformation $f_s=f+sv$.

For every such $s$, the metric defined by $f_s$ has curvature one, so
$\Delta u_s+e^{u_s}=0$. Moreover, the Liouville--Schwarzian identity \eqref{new22}
gives $T[u_s]=\mathcal S(f_s)$, where
$T[w]:=w_{zz}-\frac12w_z^2$. Differentiating these two identities at
$s=0$, we obtain
\begin{equation}\label{eq:direct-linearization-identities}
(\Delta+e^u)\mathcal V_f[v]=0,
\qquad
Q_{\mathcal V_f[v]}=D\mathcal S_f[v].
\end{equation}
Here
\begin{equation}\label{eq:direct-linearized-Schwarzian}
D\mathcal S_f[v]
:=
\frac{v_{zzz}}{f_z}
-3\frac{f_{zz}}{f_z^2}v_{zz}
+
\left(
3\frac{f_{zz}^2}{f_z^3}
-\frac{f_{zzz}}{f_z^2}
\right)v_z.
\end{equation}
Since \(f_z\neq 0\), this is a genuine third-order holomorphic linear ordinary differential operator. Then, we  obtain
\begin{align}\label{equ: ker DS}
\ker D\mathcal S_f
=
\operatorname{span}_{\mathbb C}\{1,f,f^2\}.
\end{align}
Indeed, this follows from the M\"obius invariance of \(\mathcal S\). More precisely,
\[
\mathcal S(f+t)=\mathcal S(e^t f)=\mathcal S\left(\frac{f}{1-tf}\right)=\mathcal S(f).
\]
Taking the derivative at \(t=0\), we get
$
\operatorname{span}_{\mathbb C}\{1,f,f^2\}\subset \ker D\mathcal S_f .
$
Moreover, the Wronskian of \(1,f,f^2\) is \(2f_z^3\neq0\). Since \(D\mathcal S_f\) is a third-order linear ordinary differential operator, the dimension of  kernel  has at most three, we obtain \eqref{equ: ker DS}.

Our strategy is to realize the Jacobi deformation $u+t\phi$ by an infinitesimal
deformation $f_t$ of the developing map $f$. Inspired by the variation formula
\eqref{eq:direct-linearization-identities}, we first seek a most possible variation holomorphic
function $v$ on $\Omega$ such that
\[
\mathcal V_f[v]=\phi,
\qquad
D\mathcal S_f[v]=c.
\]
We then use the two-jet of $v$ at $p_0$ to normalize the exact
Schwarzian family
\[
\mathcal S(f_t)=q_t,
\]
so that $\dot f:=\frac{\dd}{\dd t}f_t\big|_{t=0}=v$. The single-valuedness of $\phi$ and the invariance
of the round metric will imply that, for each period loop $\gamma$,
the corresponding projective monodromy satisfies
\[
\dot\varrho_0(\gamma)
\in T_{\varrho_0(\gamma)}\mathrm{PSU}(2).
\]
Thus the infinitesimal deformation determined by the exact Lam\'e path $q_t$ has the same first-order spherical character as the Jacobi deformation. Therefore, we can construct a special infinitesimal deformation satisfying \eqref{eq:specialized-character-tangency}.

\begin{theorem}\label{Key thm}
\label{thm:infinitesimal-unitary-character}
Let \(u\in X_{\rm ax}\) be a solution of \eqref{eq:mfe}, and suppose that a nonzero function \(\phi\in X_{\rm ax}\) satisfies
\[
(\Delta+e^u)\phi=0.
\]
Let \(q_t\) be the holomorphic path defined in \eqref{eq:exact-q-path-preview}. 
Given that the base point \(p_0\), let \(M_h(B(t))\) and \(M_v(B(t))\) be the corresponding monodromy matrices, as defined in Section \ref{Sec 3}. Recall $\mathcal{Z}(B)=\operatorname{tr}
\bigl(M_h(B)M_v(B)\bigr) $, then
\begin{equation}\label{eq:specialized-character-tangency}
 \frac{\dd}{\dd t}\operatorname{Im}\mathcal{Z}(B(t))\bigg|_{t=0}=0.
\end{equation}

\end{theorem}

\begin{proof}

\textbf{Step 1: Construction of $v$.}
Firstly, we construct a holomorphic function $v$ on $\Omega$ such that
\begin{equation}\label{eq:direct-representation}
\mathcal V_f[v]=\phi,
\qquad
D\mathcal S_f[v]=c.
\end{equation}
Here, we use the same notation
for the lifts of  $\phi$ to $\widetilde M$. Since $\Omega$ is simply connected and the coefficients in
\eqref{eq:direct-linearized-Schwarzian} are holomorphic, the equation
$D\mathcal S_f[v_0]=c$ has a holomorphic solution $v_0$ on $\Omega$.
Since $
\ker D\mathcal S_f
=
\operatorname{span}_{\mathbb C}\{1,f,f^2\}$, 
consider the following three elemental functions in the kernel:
\[
k_1=f,\qquad
k_2=\frac{1-f^2}{2},\qquad
k_3=\frac{i(1+f^2)}{2},
\]
and put $H_j:=\mathcal V_f[k_j]$. A direct calculation from
\eqref{eq:direct-Vf} gives
\begin{equation}\label{eq:direct-Hj}
H_1=2\frac{1-|f|^2}{1+|f|^2},\qquad
H_2=-2\frac{f+\overline f}{1+|f|^2},\qquad
H_3=2i\frac{f-\overline f}{1+|f|^2}.
\end{equation}
Writing $z=x+iy$, another direct calculation gives, 
\begin{equation}\label{eq:direct-Hj-determinant}
\det
\begin{pmatrix}
H_1&H_2&H_3\\
(H_1)_x&(H_2)_x&(H_3)_x\\
(H_1)_y&(H_2)_y&(H_3)_y
\end{pmatrix}
=
\frac{32|f_z|^2}{(1+|f|^2)^2}\ne0.
\end{equation}
Consequently, there exist unique real numbers $a_1,a_2,a_3$ such that,
for
\[
v:=v_0+a_1k_1+a_2k_2+a_3k_3,
\]
and the function $\mathcal V_f[v]$ has the same value and first derivatives
as $\phi$ at $p_0$:
\begin{align}\label{equ V_f}
\mathcal V_f[v](p_0)=\phi(p_0),\quad
\partial_x\mathcal V_f[v](p_0)=\phi_x(p_0),\quad
\partial_y\mathcal V_f[v](p_0)=\phi_y(p_0).
\end{align}
Since $k_j\in\ker D\mathcal S_f$, we still have
$D\mathcal S_f[v]=c$.

Set $\psi:=\phi-\mathcal V_f[v]$. By
\eqref{eq:direct-linearization-identities} and the hypothesis
$Q_\phi=c$, we have
\[
(\Delta+e^u)\psi=0,\qquad Q_\psi=0,
\]
while the choice of $a_1,a_2,a_3$, namely,  \eqref{equ V_f} gives
$\psi(p_0)=\psi_z(p_0)=\psi_{\bar z}(p_0)=0$. The preceding two
equations and their complex conjugates can be written as
\[
\psi_{zz}=u_z\psi_z,\qquad
\psi_{z\bar z}=-\frac14e^u\psi,\qquad
\psi_{\bar z\bar z}=u_{\bar z}\psi_{\bar z}.
\]
Let $z=z(s)$ be any $C^1$ path in $\Omega$ starting at $p_0$. Along
this path, the vector
$\Psi=(\psi,\psi_z,\psi_{\bar z})^{\mathsf T}$ satisfies
\[
\frac{\dd \Psi}{\dd s}
=
\begin{pmatrix}
0&z'&\overline{z'}\\
-\frac14e^u\overline{z'}&u_zz'&0\\
-\frac14e^uz'&0&u_{\bar z}\overline{z'}
\end{pmatrix}
\Psi.
\]
Since $\Psi(0)=0$, the uniqueness for this homogeneous first-order linear
system gives $\Psi\equiv0$ along the path. As $\Omega$ is connected,
$\psi\equiv0$ on $\Omega$. This proves
\eqref{eq:direct-representation}.

\textbf{Step 2: Solve the $f_t$ with $\dot f:=\frac{\dd}{\dd t}\big|_{t=0}f_t=v$.}
We now pass to the exact Schwarzian family. Let
$f_t:\widetilde M\to\mathbb P^1$ be the meromorphic solution of
\begin{equation}\label{eq:direct-exact-Schwarzian}
\mathcal S(f_t)=q_t=T+tc
\end{equation}
whose two-jet at $p_0$ is prescribed by
\begin{equation}\label{eq:direct-exact-jet}
J(f_t)(p_0)=J(f+tv)(p_0),
\qquad
J(h):=(h,h_z,h_{zz}).
\end{equation}
The existence of $f_t$ follows either from the Schwarzian initial-value
problem or by choosing two independent solutions of
$Y''=-q_tY/2$ whose quotient has the prescribed nondegenerate two-jet at $p_0$. Since $f_z(p_0)\ne0$, the prescribed jet remains
nondegenerate for all sufficiently small $|t|$. Standard parameter
dependence gives a $C^1$ family on compact continuation domains.

At $t=0$, the maps $f_0$ and $f$ have the same Schwarzian derivative
and the same nondegenerate two-jet at $p_0$, so $f_0=f$. Differentiating
\eqref{eq:direct-exact-Schwarzian} and
\eqref{eq:direct-exact-jet} at $t=0$ gives
\[
D\mathcal S_f[\dot f]=c,\qquad
J(\dot f)(p_0)=J(v)(p_0).
\]
Since $D\mathcal S_f[v]=c$, the difference $\dot f-v$ satisfies the
homogeneous third-order equation
$D\mathcal S_f[\dot f-v]=0$ and has zero two-jet at $p_0$. The uniqueness
for this equation therefore gives
\begin{equation}\label{eq:direct-fdot-v}
\dot f=v
\end{equation}
near $p_0$ and, by analytic continuation, along the two fixed lifted
period paths and in neighborhoods of their endpoints.

\textbf{Step 3: The projective monodromy $\dot\varrho_0(\gamma)
\in T_{\varrho_0(\gamma)}\mathrm{PSU}(2)$.}
Let
$
\varrho_t:\pi_1(M,p_0)\to\mathrm{PSL}(2,\mathbb C)
$
be the projective monodromy representation of the Schwarzian equation
$\mathcal{S}(f_t)=q_t$, characterized by
\begin{align}\label{eq:direct-projective-monodromy} 
   f_t\circ\gamma=\varrho_t(\gamma)\circ f_t. 
\end{align}
Here, $\varrho_t(\gamma)\circ f$ should be understood as follows:
\begin{align*}
    \varrho_t(\gamma)=\begin{pmatrix}a_t&b_t\\ c_t&d_t\end{pmatrix},\quad \varrho_t(\gamma)\circ f=\varrho_t(\gamma)(f) =\frac{a_tf+b_t}{c_tf+d_t}.
\end{align*}
Since \(q_t\) depends smoothly on \(t\), so does \(\varrho_t(\gamma)\). Choose a
neighborhood \(V\) of \(p_0\) such that
\(\overline V,\gamma\overline V\subset\Omega\). After shrinking the
parameter interval, all maps in \eqref{eq:direct-projective-monodromy}
are finite in the chosen affine coordinate on \(V\). Differentiating at
\(t=0\) and using \eqref{eq:direct-fdot-v}, we obtain
\begin{equation}\label{eq:direct-monodromy-linearization}
v\circ\gamma
=
\dot\varrho_0(\gamma)[f]
+
\varrho'_{0}(\gamma)(f)\,v .
\end{equation}
Here \(\dot\varrho_0(\gamma)\in T_{\varrho_0(\gamma)}\mathrm{PSL}(2,\mathbb C)\),
and 
\[
\dot\varrho_0(\gamma)[\zeta]
:=
\left.\frac{\dd}{\dd t}\right|_{t=0}\varrho_t(\gamma)(\zeta),\qquad \varrho'_0(\gamma)(\zeta):=\frac{\partial}{\partial \zeta}\varrho_0(\gamma)(\zeta).
\]

Set  \(h:=f\circ\gamma=\varrho_0(\gamma)\circ f\). Define
\[
d_\gamma
:=
v\circ\gamma-\varrho'_{0}(\gamma)(f)\,v .
\]
We claim that
\begin{equation}\label{eq:direct-zero-metric-variation}
\mathcal V_h[d_\gamma]=0 .
\end{equation}
First, direct substitution into \eqref{eq:direct-Vf} and using \eqref{eq:direct-representation} gives the
precomposition identity
\[
\mathcal V_{f\circ\gamma}[v\circ\gamma]
=
(\mathcal V_f[v])\circ\gamma=\phi\circ\gamma=\phi .
\]
The last equality follows from the fact that \(\phi\) is single-valued on \(M\), and hence its lift is invariant under deck transformations. 
On the other hand, from \eqref{eq:y1y2}, we have \(\varrho_0(\gamma)\in\mathrm{PSU}(2)\). In the chosen stereographic coordinate, the invariance
of the round metric under \(\varrho_0\) is the identity
\[
\frac{|\varrho'_{0}(\gamma)(\zeta)|^2}{(1+|\varrho_0(\gamma)(\zeta)|^2)^2}
=
\frac{1}{(1+|\zeta|^2)^2}.
\]
Applying this identity to \(f+sv\), we have
\[
\frac{|(\varrho_{0}(\gamma)(f+sv))_z|^2}{(1+|\varrho_0(\gamma)(f+sv)|^2)^2}=\frac{|\varrho'_{0}(\gamma)(f+sv)|^2|(f+sv)_z|^2}{(1+|\varrho_0(\gamma)(f+sv)|^2)^2}
=
\frac{|(f+sv)_z|^2}{(1+|f+sv|^2)^2}.
\]
Differentiating at \(s=0\) gives
\[
\mathcal V_{\varrho_0(\gamma)\circ f}
\left[\varrho'_{0}(\gamma)(f)\,v\right]
=
\mathcal V_f[v]
=\phi .
\]
Since \(\varrho_0(\gamma)\circ f=h\), subtracting the last two identities proves
\eqref{eq:direct-zero-metric-variation}.

By \eqref{eq:direct-linearization-identities}, we have
\(D\mathcal S_h[d_\gamma]=0\). Since \(h_z\ne0\), the kernel of
\(D\mathcal S_h\) is
\[
\operatorname{span}_{\mathbb C}\{1,h,h^2\}.
\]
Hence
\[
d_\gamma=A_\gamma+B_\gamma h+C_\gamma h^2
\]
for some constants \(A_\gamma,B_\gamma,C_\gamma\in\mathbb C\). A direct
calculation from \eqref{eq:direct-Vf} gives
\[
(1+|h|^2)\mathcal V_h[d_\gamma]
=
(B_{\gamma}+\overline B_{\gamma})(1-|h|^2)
+2(C_{\gamma}-\overline A_{\gamma})h
+2(\overline C_{\gamma}-A_{\gamma})\overline h .
\]
Since \(h\) is locally univalent, \(h(V)\) contains an open subset of
\(\mathbb C\). Therefore \eqref{eq:direct-zero-metric-variation} implies
\[
B_\gamma+\overline{B_\gamma}=0,
\qquad
C_\gamma=\overline{A_\gamma}.
\]
Writing \(B_\gamma=2i\alpha_\gamma\) with \(\alpha_\gamma\in\mathbb R\),
we obtain
\[
d_\gamma
=
A_\gamma+2i\alpha_\gamma h
+\overline{A_\gamma}h^2 .
\]
Set
\[
L_\gamma:=
\begin{pmatrix}
i\alpha_\gamma&A_\gamma\\
-\overline{A_\gamma}&-i\alpha_\gamma
\end{pmatrix}
\in\mathfrak{su}(2).
\]
where $\mathfrak{su}(2) $ is the Lie algebra of $\mathrm{SU}(2)$.  With the convention
\[
\begin{pmatrix}a&b\\c&d\end{pmatrix}\cdot\zeta
=
\frac{a\zeta+b}{c\zeta+d},
\]
the infinitesimal M\"obius vector field induced by \(L_\gamma\) is
\[
X_\gamma(\zeta)
=
A_\gamma+2i\alpha_\gamma\zeta
+\overline{A_\gamma}\zeta^2 .
\]
which is equivalent to
\begin{align*}
    X_{\gamma}(\zeta)=\frac{\dd}{\dd t}\Big|_{t=0}(\exp(tL_{\gamma})\cdot\zeta).
\end{align*}
Consequently, \[ d_\gamma=X_\gamma\circ h =X_\gamma\circ\varrho_0(\gamma)\circ f. \] On the other hand, \eqref{eq:direct-monodromy-linearization} gives $d_\gamma=\dot\varrho_0(\gamma)\circ f$. Since $f(V)$ is open, it follows that \[ \dot\varrho_0(\gamma) = X_\gamma\circ\varrho_0(\gamma). \]
 Therefore, as tangent vectors
to \(\mathrm{PSL}(2,\mathbb C)\) at \(\varrho_0(\gamma)\),
\[
\dot\varrho_0(\gamma)
=
\left.\frac{\dd}{\dd t}\right|_{t=0}
\big(\exp(tL_\gamma)\varrho_0(\gamma)\big).
\]
Here we have used the group structure of M\"obius transformations.   Indeed, the right-hand side acts on \(\zeta\) as
\[
\left.\frac{\dd}{\dd t}\right|_{t=0}
\big(\exp(tL_\gamma)\varrho_0(\gamma)\big)\cdot\zeta
=
X_\gamma(\varrho_0(\gamma)\cdot\zeta).
\]
Since \(L_\gamma\in\mathfrak{su}(2)\) and
\(\varrho_0(\gamma)\in\mathrm{PSU}(2)\), the curve
\[
t\mapsto \exp(tL_\gamma)\varrho_0(\gamma)
\]
lies in \(\mathrm{PSU}(2)\). Hence
\begin{equation}\label{eq:direct-projective-tangency}
\dot\varrho_0(\gamma)
\in T_{\varrho_0(\gamma)}\mathrm{PSU}(2),
\qquad
\gamma=\gamma_h,\gamma_v.
\end{equation}

\textbf{Step 4: The final conclusion.}
 Near $p_0$,
choose a $C^1$ family of square roots
$s_t=((f_t)_z)^{-1/2}$ and define
\[
\widehat Y_{1,t}:=if_ts_t,\qquad
\widehat Y_{2,t}:=is_t,\qquad
\widehat\Psi_t:=
\begin{pmatrix}
\widehat Y_{1,t}&(\widehat Y_{1,t})_z\\
\widehat Y_{2,t}&(\widehat Y_{2,t})_z
\end{pmatrix}.
\]
Then
\[
\det\widehat\Psi_t
=
(f_t)_zs_t^2=1.
\]
Recalling $\mathcal{S}(f_t)=q_t$, 
the standard Schwarzian quotient calculation shows that both $\widehat Y_{1,t}$ and $\widehat Y_{2,t}$ solve the Lam\'{e} equation \[
Y''=-\frac12q_tY
=
\bigl[\eta(\eta+1)\wp(z)+B(t)\bigr]Y.
\] Continuing $\widehat\Psi_t$ from $p_0$
defines a determinant-one fundamental matrix on $\widetilde M$.
Let $\widehat M_\gamma(t)$ denote its monodromy along $\gamma$. Its
projection to $\mathrm{PSL}(2,\mathbb C)$ is precisely
$\varrho_t(\gamma)$ because $f_t=\widehat Y_{1,t}/\widehat Y_{2,t}$.

At $t=0$, $\varrho_0(\gamma)\in\mathrm{PSU}(2)$, and the inverse image of
$\mathrm{PSU}(2)$ under
$\mathrm{SL}(2,\mathbb C)\to\mathrm{PSL}(2,\mathbb C)$ is
$\mathrm{SU}(2)$. Hence
\begin{align}\label{equ: SU2}
    \widehat M_\gamma(0)\in\mathrm{SU}(2).
\end{align}
 Since the covering map is a
local diffeomorphism, \eqref{eq:direct-projective-tangency} also gives
\begin{align}\label{equ hat M}
\dot{\widehat M}_\gamma(0)
\in
T_{\widehat M_\gamma(0)}\mathrm{SU}(2),
\qquad
\gamma=\gamma_h,\gamma_v. 
\end{align}
Although \(\widehat M_\gamma(t)\) does not belong to \(\mathrm{SU}(2)\) for $t\neq 0$, \eqref{equ: SU2} and \eqref{equ hat M} imply that, the value and velocity of \(\widehat M_\gamma(t)\) at $t=0$ agree with those of an
$\mathrm{SU}(2)$-valued path, and we may replace
$\widehat M_\gamma(t)$ by such a path when computing the first derivative at
$t=0$. The differential of the multiplication map
$\mathrm{SU}(2)\times\mathrm{SU}(2)\to\mathrm{SU}(2)$ sends tangent
vectors to tangent vectors. Therefore
\[
\left.
\frac{\dd}{\dd t}
\bigl(\widehat M_{\gamma_h}(t)
\widehat M_{\gamma_v}(t)\bigr)
\right|_{t=0}
\in
T_{\widehat M_{\gamma_h}(0)
\widehat M_{\gamma_v}(0)}\mathrm{SU}(2).
\]
Since the trace of every matrix in $\mathrm{SU}(2)$ is real, the
differential of the imaginary part of the trace vanishes on every
tangent space of $\mathrm{SU}(2)$. It follows that
\begin{equation}\label{eq:direct-trace-tangency}
\left.
\frac{\dd}{\dd t}
\operatorname{Im}\operatorname{tr}
\bigl(\widehat M_{\gamma_h}(t)
\widehat M_{\gamma_v}(t)\bigr)
\right|_{t=0}
=0.
\end{equation}

On the other hand, recalling \eqref{eq:monodromyL}, we have that $\widehat M_{\gamma_h}(t)
\widehat M_{\gamma_v}(t)=\widehat M_{\gamma_h\gamma_v}(t)$ is conjugate to $M_{\gamma_h\gamma_v}(B(t))$, and it follows from \eqref{eq:traces} that
$$\mathcal{Z}(B(t))=\tr(M_{\gamma_h\gamma_v}(B(t)))=\tr\widehat M_{\gamma_h\gamma_v}(t)=\tr\bigl(\widehat M_{\gamma_h}(t)
\widehat M_{\gamma_v}(t)\bigr).$$
Combining this identity with
\eqref{eq:direct-trace-tangency} gives
\[
\left.
\frac{\dd}{\dd t}\operatorname{Im} \mathcal{Z}(B(t))
\right|_{t=0}
=0,
\]
as required. The proof is complete.
\end{proof}

\begin{theorem}\label{thm:ax-nondeg}
For every $\rho\in(8\pi(n-1),8\pi n)$ and every axisymmetric solution
$u$, the kernel satisfies
\begin{equation}\label{eq:ax-kernel}
 \ker(\Delta+e^u)\cap X_{\rm ax}=\{0\}.
\end{equation}
\end{theorem}

\begin{proof}
Assume for contradiction that there exists
$0\neq\phi\in X_{\rm ax}$ such that $(\Delta+e^u)\phi=0$.
By Propositions \ref{prop:Qconstant} and \ref{prop:c-nonzero},
\[
Q_\phi=\phi_{zz}-u_z\phi_z=c
\qquad\text{for some }c\in\mathbb R\setminus\{0\}.
\]
Let $B_0\in\mathbb R$ be the accessory parameter of $u$ and set
$B(t)=B_0-\frac c2t$. By Theorem \ref{Key thm},
\[
0=\frac{\dd}{\dd t}\mathrm{Im} \mathcal{Z}(B(t))\bigg|_{t=0}
=-\frac c2\frac{\dd}{\dd B}\mathrm{Im} \mathcal{Z}(B)\bigg|_{B=B_0}.
\]
Since $c\neq0$, this gives
\[
\frac{\dd}{\dd B}\mathrm{Im} \mathcal{Z}(B)\bigg|_{B=B_0}=0,
\]
a contradiction with
Theorem \ref{thm:transversality}.
\end{proof}

\section{Continuation and proof of the main theorem}\label{Sec 6}

In this final section, we complete the proofs of Theorems \ref{thm:main} and \ref{thm:Lame}.

\begin{proof}[Proof of Theorem \ref{thm:main}]  Let $n\in\mathbb{N}_{\geq 1}$ and $\rho\in (8\pi(n-1),8\pi n)$.  We need to show that the number of axisymmetric solutions is constant on
	the whole interval.   Recall the Green function $G(z)$ defined in \eqref{eq:Green}. Since $E_{\ii b}$ is a rectangular torus,
the Green function $G(z)$ is axisymmetric. Writing
\[
w:=u+\rho G,\qquad h_\rho:=e^{-\rho G},
\]
the singular Liouville equation \eqref{eq:mfe} becomes
\begin{equation}\label{eq:regularized-equation}
F(\rho,w):=
\Delta w+h_\rho e^w-\frac{\rho}{|E_{ib}|}=0.
\end{equation}
For every $0<\beta<1$, this defines a $C^1$ map
\[
F:(8\pi(n-1),8\pi n)\times C_{\rm ax}^{2,\beta}(E_{\ii b})
\longrightarrow C_{\rm ax}^{0,\beta}(E_{\ii b}).
\]
where $C_{\rm ax}^{k,\beta}(E_{\ii b})=C^{k,\beta}(E_{\ii b})\cap X_{\rm ax}$.   At a solution,
\[
D_wF(\rho,w)=\Delta+e^u.
\]
This operator is self-adjoint, Fredholm of index zero, and commutes
with the rectangular reflections. By Theorem \ref{thm:ax-nondeg}, its restriction to
the axisymmetric subspace has trivial kernel and is therefore an
isomorphism. The implicit function theorem consequently gives a unique
local $C^1$ solution branch through every axisymmetric solution.

We next record the properness. Let
\[
K\Subset(8\pi(n-1),8\pi n).
\]
By Theorem \ref{thm A},
\[
\|w\|_{C^2(E_{\ii b})}
=\|u+\rho G\|_{C^2(E_{\ii b})}
\le C_K
\]
for every solution with $\rho\in K$. Since
\[
h_\rho(z)=|z|^{\rho/(2\pi)}e^{-\rho R(z)}
\]
near $0$ and
$\inf_{\rho\in K}\rho/(2\pi)\geq c_{K}>0$, the family $h_\rho$ is uniformly
bounded in $C^{0,\beta_1}(E_{\ii b})$ for $0<\beta_1<\min\{c_K, 1\}$.
Equation
\eqref{eq:regularized-equation} and the Schauder estimates therefore
give
\[
\|w\|_{C^{2,\beta_1}(E_{\ii b})}\le C_{K,\beta_1}.
\]
Hence the solution set over $K$ is compact in
$C_{\rm ax}^{2,\beta}(E_{\ii b})$ by taking $0<\beta<\beta_1<1$.

It follows that the number
\[
N_{\rm ax}(\rho)
:=
\#\left\{
w\in C_{\rm ax}^{2,\beta}(E_{\ii b}):
F(\rho,w)=0
\right\}
\]
is locally constant. Indeed, each fiber is compact and, by the
implicit function theorem, discrete, hence finite. The finitely many
local solution branches through the fiber over $\rho_0$ contain every
solution for all $\rho$ sufficiently close to $\rho_0$; otherwise, the
properness would produce an additional limiting solution in the fiber
over $\rho_0$. Thus $N_{\rm ax}$ is a constant on the connected interval
$(8\pi(n-1),8\pi n)$.

By the results mentioned in Remark \ref{rmk1-3}-(2), there exists
$\rho_0\in(8\pi(n-1),8\pi n)$ such that
equation \eqref{eq:mfe} has exactly $n$ solutions, which are all axisymmetric by Lemma \ref{lem:reflection-rigidity}. Therefore
\[
N_{\rm ax}(\rho)=n
\qquad
\text{for every }\rho\in(8\pi(n-1),8\pi n).
\]
Finally, Lemma \ref{symmetric lem} shows that every solution is axisymmetric. Hence
equation \eqref{eq:mfe} has exactly $n$ solutions for every
$\rho\in(8\pi(n-1),8\pi n)$.
\end{proof}

\begin{proof}[Proof of Theorem \ref{thm:Lame}]
This theorem follows directly from Theorem \ref{thm:main} and Corollary \ref{coro:11}.
\end{proof}

\subsection*{Acknowledgments}  Z. Chen is supported by National Key R\&D Program of China (No. 2023YFA1010002) and NSFC (No. 12222109). 
S. Zhang is supported by the Postdoctoral Fellowship Program and China Postdoctoral Science Foundation under Grant Numbers BX20250062 and 2026M793381.
The authors acknowledge the use of AI tools. The authors assume full responsibility for the mathematical validity, accuracy, and integrity of the proofs presented in the manuscript.


\begin{thebibliography}{99}

\bibitem{BCLT} D.~Bartolucci, C.-C.~Chen, C.-S.~Lin, G.~Tarantello, \emph{Profile of blow-up solutions to mean
field equations with singular data}, Comm. Partial Differential Equations \textbf{29} (2004),
1241--1265. 

\bibitem{BMM} D.~Bartolucci, F.~De Marchis and A.~Malchiodi, \emph{Supercritical conformal metrics on surfaces with conical singularities}, Int. Math. Res. Not. \textbf{2011} (2011), 5625--5643.

\bibitem{BartolucciTarantelloCMP}
D.~Bartolucci and G.~Tarantello,
\emph{Liouville type equations with singular data and their applications
to periodic multivortices for the electroweak theory},
Comm. Math. Phys. \textbf{229} (2002), no.~1, 3--47.

\bibitem{BM} H.~Brezis and F.~Merle, \emph{Uniform estimates and blow-up behavior for
solutions of $-\Delta u=V(x)e^u$ in two dimensions}, Comm. Partial Differential Equations \textbf{16} (1991), 1223--1253.

\bibitem {CLMP}
E.~Caglioti, P. L.~Lions, C.~Marchioro and M.~Pulvirenti,
\emph{A special class of stationary flows for two-dimensional Euler
equations: a statistical mechanics description}, Comm. Math. Phys.
\textbf{143} (1992), 501--525.

\bibitem{CM} A.~Carlotto and A.~Malchiodi, \emph{Weighted barycentric sets and singular Liouville equations on compact surfaces}, J. Funct. Anal. \textbf{262} (2012), 409--450. 

\bibitem{ChaiLinWang}
C.-L.~Chai, C.-S.~Lin, and C.-L.~Wang,
\emph{Mean field equations, hyperelliptic curves and modular forms: I},
Camb. J. Math. \textbf{3} (2015), no.~1--2, 127--274.

\bibitem{ChenLinCPAM}
C.-C.~Chen and C.-S.~Lin,
\emph{Mean field equation of Liouville type with singular data:
topological degree},
Comm. Pure Appl. Math. \textbf{68} (2015), no.~6, 887--947.



\bibitem {CLW4}C.-C.~Chen, C.-S.~Lin and G. Wang; \emph{Concentration
phenomena of two-vortex solutions in a Chern-Simons model}, Ann. Scuola Norm.
Sup. Pisa Cl. Sci. (5) \textbf{3} (2004), 367--397.

\bibitem{ChenLinAJM}
Z.~Chen and C.-S.~Lin,
\emph{Sharp nonexistence results for curvature equations with four
singular sources on rectangular tori},
Amer. J. Math. \textbf{142} (2020), no.~4, 1269--1300.

\bibitem{ChenLinJDG}
Z.~Chen and C.-S.~Lin,
\emph{Exact number and non-degeneracy of critical points of multiple Green
functions on rectangular tori},
J. Differential Geom. \textbf{118} (2021), no.~3, 457--485.

\bibitem{DKM} M.~del Pino, M.~Kowalczyk and M.~Musso, \emph{Singular limits in Liouville-type equations}, Cal. Var.
Partial Differential Equations \textbf{24} (2005), 47--81.



\bibitem{Eastham}
M.~S.~P.~Eastham,
\emph{The Spectral Theory of Periodic Differential Equations},
Scottish Academic Press, Edinburgh, 1973.

\bibitem{Eremenko} A.~Eremenko,  \emph{Metrics of constant positive curvature with four conic singularities on the sphere},
Proc. Amer. Math. Soc. \textbf{148} (2020), 3957--3965.

\bibitem{EG}
A.~Eremenko and A.~Gabrielov, \emph{Spherical Rectangles},
Arnold Math. J. \textbf{2} (2016), 463--486.

\bibitem{EMP}
A.~Eremenko, G.~Mondello, and D.~Panov,
\emph{Moduli of spherical tori with one conical point},
Geom. Topol. \textbf{27} (2023), no.~9, 3619--3698.

\bibitem {GW-Acta}F.~Gesztesy and R.~Weikard; \emph{Picard potentials and Hill's
equation on a torus}. Acta Math. \textbf{176} (1996), 73--107.

\bibitem{GilbargTrudinger}
D.~Gilbarg and N.~S.~Trudinger,
\emph{Elliptic Partial Differential Equations of Second Order},
Classics in Mathematics, Springer-Verlag, Berlin, 2001.

\bibitem{GoldmanFricke}
W.~M. Goldman,
\emph{Trace coordinates on Fricke spaces of some simple
hyperbolic surfaces},
in Handbook of Teichm\"uller Theory, Vol.~II,
IRMA Lect. Math. Theor. Phys. \textbf{13},
European Mathematical Society, Z\"urich, 2009,
611--684.

\bibitem{HornJohnsonTopics}
R.~A. Horn and C.~R. Johnson,
\emph{Topics in Matrix Analysis},
Cambridge University Press, Cambridge, 1991.


\bibitem{JWY}  A.~Jevnikar, J.~Wei and W.~Yang, \emph{On the topological degree of the mean field equation with two parameters}, Indiana Univ. Math. Journal \textbf{67} (2018), 29--88.

\bibitem{Kuo} T.-J. Kuo, \emph{Sharp results for spherical metric on flat tori with conical angle $6\pi$ at two symmetric points},
J. Differential Geom. \textbf{133} (2026), 275--334.

\bibitem{LinJDG} C.-S.~Lin, \emph{Spherical metrics with one singularity and odd integer angle on flat tori, I}. J. Differ. Geom. \textbf{132} (2026),  321--382.

\bibitem{LinWangAnnals}
C.-S.~Lin and C.-L.~Wang,
\emph{Elliptic functions, Green functions and the mean field equations
on tori},
Ann. of Math. (2) \textbf{172} (2010), no.~2, 911--954.


\bibitem{LinWang}
C.-S.~Lin and C.-L.~Wang,
\emph{On the minimality of extra critical points of Green functions on
flat tori},
Int. Math. Res. Not. 2017 (2017), no.~18, 5591--5608.

\bibitem{LinWangJEP}
C.-S.~Lin and C.-L.~Wang,
\emph{Mean field equations,
Hyperelliptic curves, and Modular forms: II. With Appendix A written by Y.C.
Chou}, J. \'{E}c. polytech. Math. \textbf{4} (2017), 557--593.


\bibitem {LY}C.-S.~Lin and S.~Yan, \emph{Existence of bubbling solutions for
Chern-Simons model on a torus.} Arch. Ration. Mech. Anal. \textbf{207} (2013), 353--392.

\bibitem{MZ-IMRN} R.~Mazzeo and X.~Zhu, \emph{Conical metrics on Riemann surfaces, II: Spherical metrics},
Int. Math. Res. Not. \textbf{2022} (2022), 9044--9113.

\bibitem{MP-IMRN} R. Mondello and D. Panov, \emph{Spherical metrics with conical singularities on
a 2-sphere: angle constraints}, Int. Math. Res. Not. \textbf{2016} (2016), 4937--4995.

\bibitem{MP-GAFA} G. Mondello and D. Panov, \emph{Spherical surfaces with conical points: systole inequality
and moduli spaces with many connected components}, Geom. Funct. Anal. \textbf{29} (2019),
1110--1193.

\bibitem {NT1}M.~Nolasco and G.~Tarantello; \emph{Double vortex condensates
in the Chern-Simons-Higgs theory.} Calc. Var. PDE. \textbf{9} (1999), 31--94.

\bibitem{MagnusWinkler}
W.~Magnus and S.~Winkler,
\emph{Hill's Equation},
Interscience Tracts in Pure and Applied Mathematics, No.~20,
Interscience Publishers, New York, 1966.

\bibitem{MR} A.~Malchiodi and D.~Ruiz, \emph{New improved Moser-Trudinger inequalities and singular Liouville equations on compact surfaces}, Geom. Funct. Anal. \textbf{21} (2011), 1196--1217.

\bibitem{McLaughlinNabelek}
K.~T.-R.~McLaughlin and P.~V.~Nabelek,
\emph{A Riemann--Hilbert problem approach to infinite gap Hill's operators
and the Korteweg--de Vries equation},
Int. Math. Res. Not. IMRN 2021 (2021), no.~2, 1288--1352.


\bibitem{Tarantello} G.~Tarantello, \emph{A Harnack inequality for Liouville-type equations with singular sources}, Indiana
Univ. Math. J. \textbf{54} (2005), 599--615. 

\bibitem{WWX} Z.~Wei, Y.~Wu and B.~Xu, 
\emph{Geometric structure and existence of reducible spherical conical metrics},
Math. Ann. \textbf{395} (2026), Paper No. 2, 48 pp.

\end{thebibliography}
\end{document}